\documentclass[aap, noinfoline]{imsart}
\RequirePackage[OT1]{fontenc}
\RequirePackage{amsthm,amsmath, amssymb}
\usepackage{graphicx} 
\usepackage[all]{xy}
\usepackage{color}
\RequirePackage[numbers]{natbib}
\usepackage{mathrsfs}

\RequirePackage[colorlinks,citecolor=blue,urlcolor=blue]{hyperref}

\setattribute{journal}{name}{}

\startlocaldefs
\numberwithin{equation}{section}
\theoremstyle{plain}
\newtheorem{thm}{Theorem}[section]
\newtheorem{lem}[thm]{Lemma}
\newtheorem{cor}[thm]{Corollary}
\newtheorem{prop}[thm]{Proposition}

\theoremstyle{remark}
\newtheorem{rem}[thm]{Remark}
\newtheorem{ex}[thm]{Example}

\theoremstyle{definition}
\newtheorem{defn}[thm]{Definition}

\newcommand{\bF}{{\bf F}}

\newcommand{\bN}{{\bf N}}

\newcommand{\bR}{{\bf R}}

\newcommand{\bZ}{{\bf Z}}
\newcommand{\C}{{\mathbb{C}}}

\newcommand{\E}{{\mathbb{E}}}

\newcommand{\K}{{\mathbb{K}}}
\newcommand{\N}{{\mathbb{N}}}
\newcommand{\R}{{\mathbb{R}}}

\newcommand{\Z}{{\mathbb{Z}}}

\newcommand{\cA}{{\mathcal A}}

\newcommand{\cF}{{\mathcal F}}

\newcommand{\cI}{{\mathcal I}}

\newcommand{\cM}{{\mathcal M}}
\newcommand{\cN}{{\mathcal N}}

\newcommand{\cP}{{\mathcal P}}

\newcommand{\cR}{{\mathcal R}}
\newcommand{\cS}{{\mathcal S}}

\newcommand{\cW}{{\mathcal W}}
\newcommand{\cX}{{\mathcal X}}

\newcommand{\cech}{\mbox{\rm{\v{C}ech} }}

\DeclareMathOperator{\rank}{rank}
\DeclareMathOperator{\image}{im}
\DeclareMathOperator{\kernel}{ker}

\newcommand{\id}{\mathop{\mathrm{id}}\nolimits}

\newcommand{\diam}{{\rm diam}}

\DeclareMathOperator{\support}{supp}

\newcommand\mfA{\mathscr{A}}
\newcommand\bdd{\mathscr{B}}
\newcommand\mfB{\mathscr{B}}
\newcommand\mfC{\mathscr{C}}

\newcommand\finite{\mathscr{F}}

\newcommand\bra{\langle}
\newcommand\ket{\rangle}

\newcommand\La{\Lambda}
\newcommand\PP{\mathbb{P}}

\newcommand\Rbar{\overline{\bR}}

\newcommand\fp{\Xi}
\newcommand\phq{\mathrm{PH}_q}
\newcommand\eps{\varepsilon}

 \newcommand\lref[1]{Lemma~\ref{#1}}
 \newcommand\pref[1]{Proposition~\ref{#1}}
 \newcommand\cref[1]{Corollary~\ref{#1}}

\newcommand\U{\mathbb U}
\newcommand\V{\mathbb V}

\newcommand\Var{{\rm Var\,}}
\newcommand\Ex{{\mathbb E\,}}
\newcommand\dto{\overset{d}{\to}}
\newcommand\vto{\overset{v}{\to}}

\newcommand\one{\mathbf 1}
\newcommand\rM{\mathfrak M}
\newcommand\rN{\mathfrak N}

\endlocaldefs

\begin{document}

\begin{frontmatter}
\title{Limit theorems for persistence diagrams}
\runtitle{Limit theorems for persistence diagrams}

\begin{aug}
\author{\fnms{Trinh Khanh Duy}\thanksref{m1}\ead[label=e1]{trinh@imi.kyushu-u.ac.jp}},
\author{\fnms{Yasuaki Hiraoka}\thanksref{m2}\ead[label=e2]{hiraoka@wpi-aimr.tohoku.ac.jp}}
\and
\author{\fnms{Tomoyuki Shirai}\thanksref{m1}
\ead[label=e3]{shirai@imi.kyushu-u.ac.jp}}

\runauthor{T. K. Duy, Y. Hiraoka and T. Shirai}

\affiliation{Kyushu University\thanksmark{m1} and Tohoku University\thanksmark{m2}}

\address{
Institute of Mathematics for Industry\\
Kyushu University\\
744, Motooka, Nishi-ku\\ Fukuoka, 819-0395, Japan\\
\printead{e1}\\
\phantom{E-mail:\ }\printead*{e3}}

\address{Advanced Institute for Materials Research\\
Tohoku University\\
2-1-1, Katahira, Aoba-ku\\ Sendai, 980-8577, Japan\\
\printead{e2}}
\end{aug}
\vspace{0.3cm}
\today
\begin{abstract}
The persistent homology of a stationary point process on ${\bf R}^{N}$ is studied in this paper. As a generalization of continuum percolation theory, we study higher dimensional topological features of the point process such as loops, cavities, etc. in a multiscale way. The key ingredient is the persistence diagram, which is an expression of the persistent homology. We prove the strong law of large numbers for persistence diagrams as  the window size tends to infinity and give a sufficient condition for the limiting persistence diagram to have the full support. We also discuss a central limit theorem for persistent Betti numbers. 
\end{abstract}



\end{frontmatter}

\section{Introduction}\label{sec:intro}
\subsection{Background}
The prototype of this work dates back to the random geometric graphs. 
In those original settings, a set $V$ of points is randomly scattered in
a space according to some probability distribution, and a graph with the
vertices $V$ is constructed by assigning edges whose distances are less
than a certain threshold value $r \ge 0$. 
Then, some characteristic features in the graph such as connected components and loops are broadly and thoroughly studied (see e.g. \cite{p}). Furthermore, the random geometric graphs provide mathematical models for applications such as 
mobile wireless networks \cite{wireless1,wireless2}, epidemics \cite{epidemics1}, and so on.

Recently, the concept of random topology has emerged and rapidly grown as a higher dimensional generalization of random graphs \cite{bk,kahle2}. One of the simple models studied in random topology is simplicial complex, which is given by a collection of subsets closed under inclusion. Obviously, a graph is regarded as a one dimensional simplicial complex consisting of singletons as vertices and doubletons as edges.

In geometric settings, a simplicial complex is built over randomly 
distributed points in a space by a certain rule respecting the nearness
of multiple points, like random geometric graphs. Two standard
simplicial complex models constructed from the points are \v{C}ech
complexes and Rips complexes, which are also determined by a threshold
value $r$ measuring the nearness of points. Then, in such an
extended geometric object, it is natural to study higher dimensional
topological features such as cavities (2 dim.) and more general
$q$-dimensional holes, beyond connected components (0 dim.) and loops (1
dim.).  

In algebraic topology, $q$-dimensional holes are usually characterized by using the so-called homology. 
Here, the $q$th homology of a simplicial complex is given by a
vector space and its dimension is called the Betti number which counts
the number of $q$-dimensional holes. 
Hence, in the setting of random simplicial complexes, the Betti numbers
become random variables through a random point configuration, and studying the asymptotic behaviors of the randomized Betti numbers is a significant problem for understanding global topological structures embedded in the random simplicial complexes (e.g., \cite{d,kahle1,oa,ya,ysa}).

On the other hand, another type of generalizations has been recently attracting much attention in applied topology. In that setting, we are interested in how persistent the holes are for changing the threshold parameter $r\in\bR$. Namely, we deal with one parameter filtration of simplicial complexes obtained by increasing the parameter $r$ and characterize robust or noisy holes in that filtration. The persistent homology \cite{elz,zc}
is a tool invented for this purpose, and especially, its expression
called persistence diagram is now applied to a wide variety of applied
areas (see e.g., \cite{carlsson,protein,pnas,nanotechnology, natcom}). 
From this point of view, there have been some works on a functional of
persitence diagram, called lifetime sum or total persistence, for 
random complexes (that are not geometric in the sense above)
such as Linial-Meshulam
processes and random cubical complexes (e.g., \cite{hs_frieze, hs_tutte, ht}).  

Therefore, it is natural to further extend the results on random 
geometric simplicial complexes to this generality, and the purpose of
this paper is to show several these extensions. In particular, we are
interested in asymptotic behaviors of persistence diagrams
themselves defined on staionary point processes.  
These subjects are mathematically meaningful in its own right, but are also interesting for practical applications. 
For example, the paper \cite{pnas} studies topological and geometric structures of atomic configurations in glass materials by comparing persistence diagrams with those of disordered states. 
By regarding atomic configurations in disordered states as random point processes, further understanding of those persistence diagrams will be useful for characterizing geometry and topology of glass materials, which is one of the important research topics in current physics.

\subsection{Prior work}

Let $\Phi$ be a stationary point process on $\bR^{N}$ with all finite
moments, i.e., 
\begin{equation}
	\Ex[\Phi(A)^k ] < \infty, \text{ for all bounded Borel sets $A$
	and any $k = 1,2,\dots$}. 
\label{eq:moment_condition} 
\end{equation}
Here $\Phi(A)$ denotes the number of points in $A$. For simplicity,
we always assume that $\Phi$ is \textit{simple}, i.e.,
\[
 \PP(\text{$\Phi(\{x\}) \le 1$ for every $x \in \bR^{N}$}) = 1. 
\]
We denote by $\Phi_{\La_L}$ the restriction of $\Phi$ on $\La_L =
[-\frac{L}{2}, \frac{L}{2})^{N}$.  

Let $C(\Phi_{\Lambda_L},r)$ be the \v{C}ech complex built over the points
$\Phi_{\Lambda_L}$ with a parameter $r >0$ (see Section \ref{sec:gm} for
the definition). Limiting behaviors of the Betti numbers
$\beta_q(C(\Phi_{\La_L}, r)) \ (q=1,2,\dots, N-1)$
have been widely investigated \cite{ya, ysa}.
Among them, we here restate the most related results to this paper.

\begin{thm}[{\cite[Lemma 3.3 and Theorem 3.5]{ysa}}]\label{thm:ysa1}
Assume that $\Phi$ is a stationary point process on $\bR^{N}$ having all
 finite moments. Then there exists a constant $\hat \beta^{r}_q \geq 0$
 such that  
\[
	\frac{\E[\beta_q(C(\Phi_{\Lambda_L},r))] }{L^{N}} \to \hat \beta^{r}_q \text{ as } L \to \infty.
\] 
In addition, if $\Phi$ is ergodic, then 
\[
	\frac{\beta_q (C(\Phi_{\Lambda_L},r)) }{L^{N}} \to \hat \beta^{r}_q  \text{ almost surely as } L \to \infty.
\] 
\end{thm}

\begin{thm}[{\cite[Theorem 4.7]{ysa}}]\label{thm:ysa2}
	Assume that $\Phi$ is a homogeneous Poisson point process on
 $\bR^{N}$ with unit intensity. Then, there exists a constant $\sigma_{r}^2 > 0$ such that 
	\[
		\frac{\beta_q (C(\Phi_{\Lambda_L},r)) - \E[\beta_q (C(\Phi_{\Lambda_L},r)) ]}{L^{N/2}} \dto \cN(0, \sigma_{r}^2)\text{ as } L \to \infty.
	\]
Here $\cN(\mu, \sigma^2)$ denotes the normal distribution with mean $\mu$ and variance $\sigma^2$, and $\dto$ denotes the convergence in distribution of random variables.
\end{thm}


The purpose of this paper is to extend Theorem~\ref{thm:ysa1} to the
setting on persistence diagrams and Theorem~\ref{thm:ysa2} to persistent
Betti numbers. 

\subsection{Main results}

In this paper, we study the following simplicial complex model for the point process $\Phi$ which is a generalization of the \v{C}ech complex and the Rips complex. 

Let $\finite(\bR^{N})$ be the collection of all finite (non-empty) subsets in
$\bR^{N}$. We can identify $\finite(\bR^{N})$ with the set
$\sqcup_{k=1}^{\infty}(\bR^{N})^k \slash \sim$, where $\sim$ is the
equivalence relation induced by permutations of coordinates. For a function
$f$ on $\finite(\bR^{N})$, there exists a permutation invariant function
$\tilde{f}_k$ on $(\bR^{N})^k$ for each $k \ge 1$ such that
$f(\{x_1,\dots, x_k\}) = \tilde{f}_k(x_1,\dots,x_k)$.
We say that $f$ is measurable
if so is $\tilde{f}_k$ on $(\bR^{N})^k$ for each $k \ge 1$. 

Let $\kappa \colon \finite(\bR^{N})\rightarrow[0,\infty]$ be a measurable
function satisfying 
\begin{itemize}
\item[(K1)] $0 \le \kappa(\sigma) \le \kappa(\tau)$, if $\sigma$ is a subset of $\tau$;
\item[(K2)] $\kappa$ is translation invariant, i.e., $\kappa(\sigma + x) = \kappa(\sigma)$ for any $x\in \bR^{N}$,
where $\sigma + x := \{y + x : y \in \sigma\}$;
\item[(K3)] there is an increasing function 
       $\rho \colon [0, \infty] \to [0, \infty]$ with
       $\rho(t) < \infty$ for $t < \infty$ such that
	\[
		\|x - y\| \le \rho(\kappa(\{x,y\})),
	\]
	where $\|x\|$ denotes the Euclidean norm in $\bR^{N}$. 
\end{itemize}
Without loss of generality, we can assume $\kappa(\{x\})
= 0$  because of the translation invariance.

\begin{ex}\label{ex:kappas}
Two important examples of $\kappa$ which we have in mind are 
\begin{align}
\kappa_C(\{x_0,x_1,\dots,x_q\}) &= \inf_{w \in \bR^{N}} \max_{0 \le i \le
q} \|x_i - w\|, \\   
  \kappa_R(\{x_0,x_1,\dots,x_q\}) &= \max_{0 \le i < j \le q}
\frac{\|x_i - x_j\|}{2},  
\end{align}
which define the \cech filtration $\C(\Phi)=\{C(\Phi,t)\}_{t\geq 0}$ and the
 Rips filtration $\R(\Phi)=\{R(\Phi,t)\}_{t\geq 0}$, respectively.
Both $\kappa$'s satisfy Assumption (K3) with $\rho(t) = 2t$.  
 See also Section~\ref{sec:gm} for these filtrations. 
\end{ex}

Given such a function $\kappa$, we construct a filtration $\K(\fp) =
\{K (\fp, t) : 0\leq t<\infty\}$ of simplicial complexes from
a finite point configuration $\Xi \in \finite(\bR^{N})$ by 
\begin{align}\label{eq:sc_model}
	K(\fp, t) = \{ \sigma \subset \fp : \kappa(\sigma) \le t
\}, 
\end{align}
i.e., $\kappa(\sigma)$ is the birth time of a simplex $\sigma$ in the
filtration $\K(\fp)$. 
Although we do not explicitly show the dependence on $\kappa$ in the
notation $\K(\fp)$ because the function $\kappa$ is fixed in the paper, we
here call it the \textit{$\kappa$-filtration} over $\fp$. 

For the filtration $\K(\fp)$,
we denote its $q$th persistence diagram by 
\[
D_{q}(\fp) = \{(b_i,d_i)\in\Delta : i=1,\dots,n_q\},
\] 
which is given by a multiset on $\Delta=\{(x,y)\in\Rbar^2  : 0\leq
x< y \le \infty\}$ determined from the unique decomposition of the persistent
homology (see (\ref{eq:pd}) for the definition).
The pair $(b_i,d_i)$ indicates the persistence of the $i$th
homology class, i.e., it appears at $b_i$ and disappears at $d_i$, and
$d_i = \infty$ means that the $i$th homology class persists forever.

In this paper, we deal with the persistence diagram $D_q(\fp)$
as the counting measure 
\[
	\xi_{q}(\fp) = \sum_{(b_i,d_i)\in D_{q}(\fp)} \delta_{(b_i, d_i)},
\]
rather than as a multiset,
where $\delta_{(x,y)}$ is the Dirac measure at $(x,y)\in\Rbar^2$.

For each $L >0$, we define a random filtration built over the points $\Phi_{\Lambda_L}$ and denote it by 
$\K(\Phi_{\Lambda_L})=\{K(\Phi_{\Lambda_L},t)\}_{t\geq 0}$.
We write $\xi_{q,L}$ for the point process $\xi_q(\Phi_{\La_L})$ and
$\E[\xi_{q,L}]$ for its mean measure (see Section~\ref{sec:rm} for the
precise definition of mean measure).
 
\begin{ex}
The top three panels in Figure \ref{fig:pppd} show point processes with
 negative (Ginibre), zero (Poisson), and positive (Poisson cluster)
 correlations, respectively. 
All point processes consist of $1,000,000$ points with the density
 $1/2\pi$, and only restricted areas of them are visualized. The bottom shows the corresponding normalized persistence
 diagrams $\xi_{1,L}/L^2$ of the {\v C}ech filtrations applied to the
 above, respectively.
\begin{figure}[h]
 \begin{center}
  \includegraphics[width=0.32\hsize]{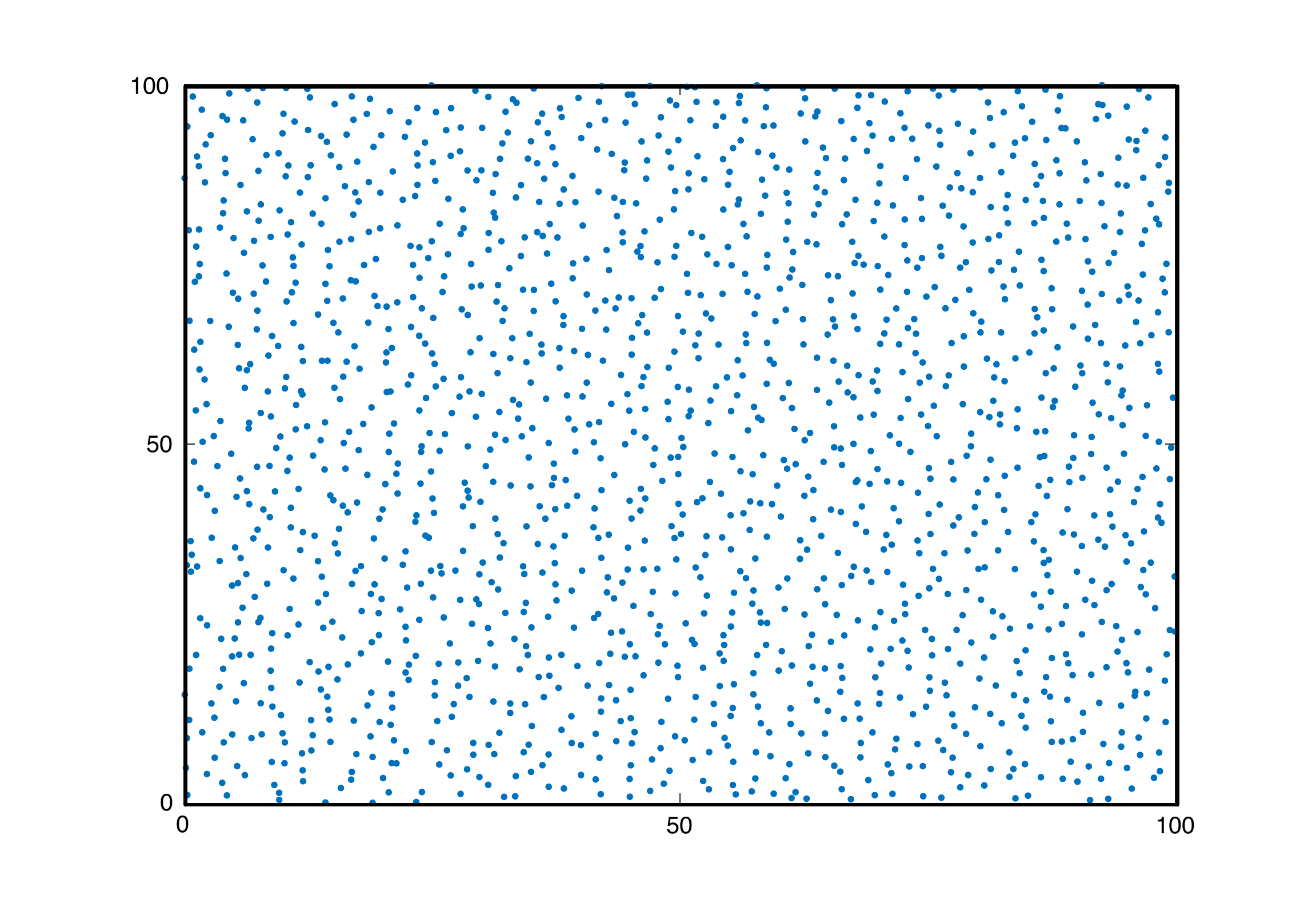}  
  \includegraphics[width=0.32\hsize]{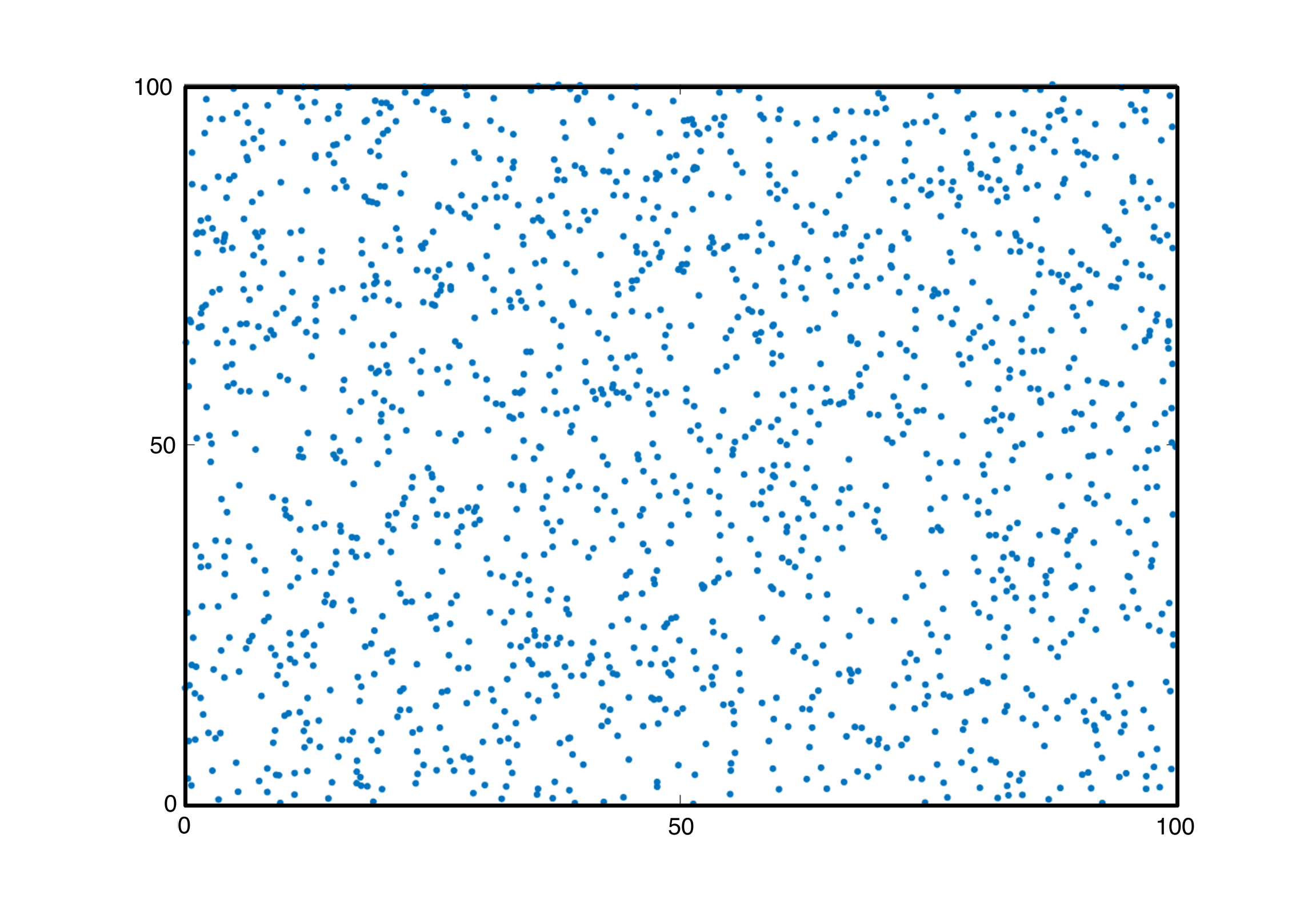}    
  \includegraphics[width=0.32\hsize]{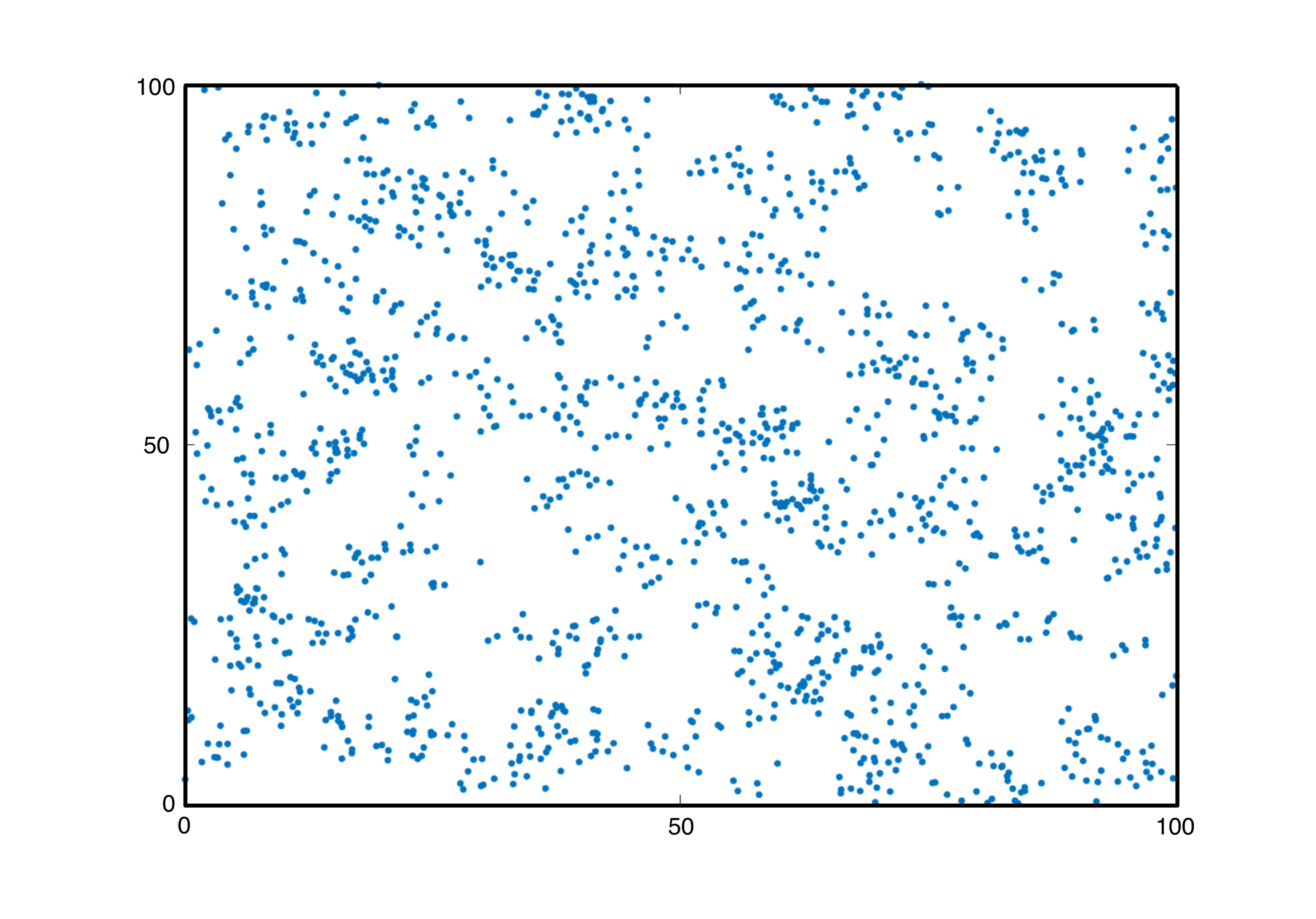}\\
  \includegraphics[height=0.32\hsize,width=0.32\hsize]{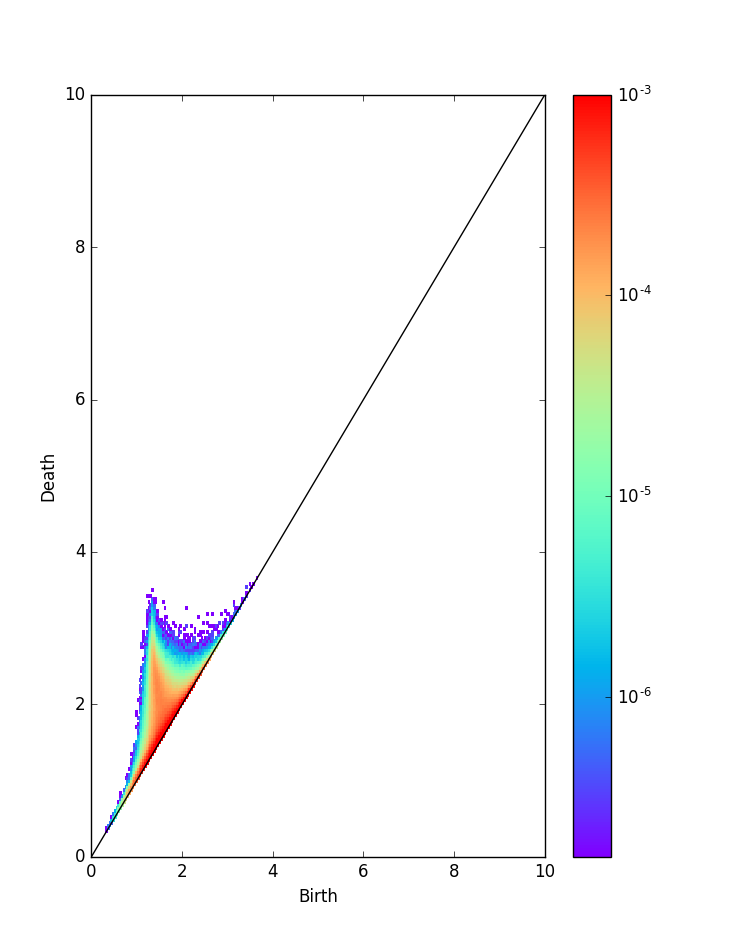}  
  \includegraphics[height=0.32\hsize,width=0.32\hsize]{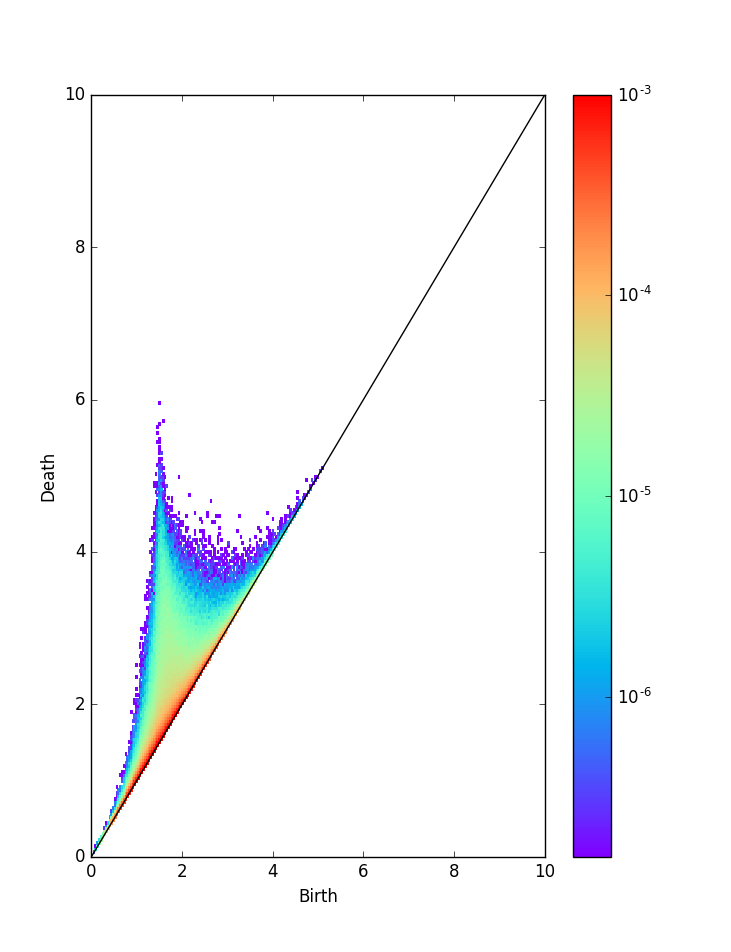}    
  \includegraphics[height=0.32\hsize,width=0.32\hsize]{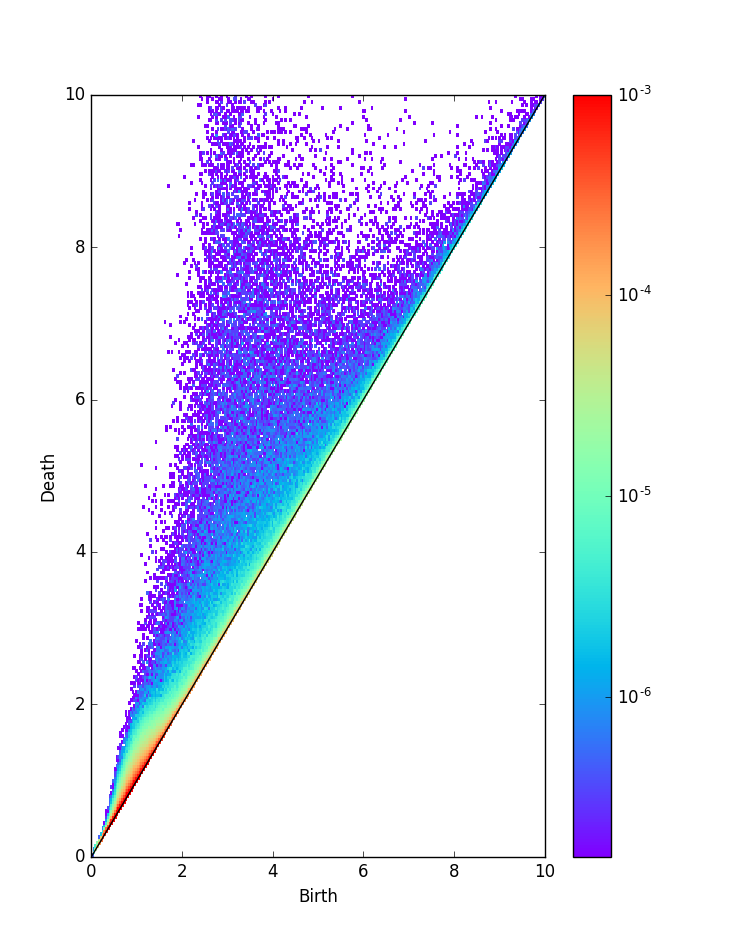}
 \end{center}
  \caption{Top: Point processes with negative (Ginibre), zero (Poisson),
 and positive (Poisson cluster) correlations. In these three point
 processes, the number of points  and the density are set to be
 $1,000,000$ and $1/2\pi$, respectively.  Bottom: The normalized
 persistence diagrams $\xi_{1,L}/L^2$ of the above. } 
  \label{fig:pppd}
\end{figure}
\end{ex}

One of the main results in this paper is as follows. 
\begin{thm}\label{thm:LLN-PD}
Assume that $\Phi$ is a stationary point process on $\bR^{N}$ having all finite moments. Then  
there exists a unique Radon measure $\nu_{q}$ on $\Delta$ such that
\begin{equation}
\frac{1}{L^{N}}\E[\xi_{q,L}] \vto \nu_{q} \text{ as } L \to \infty.
 \label{eq:conv_main1}
\end{equation}
Here $\vto$ denotes the vague convergence of measures on $\Delta$.
In addition, if $\Phi$ is ergodic, then almost surely, 
\begin{equation}
\frac{1}{L^{N}}\xi_{q,L} \vto \nu_{q} \text{ as } L \to \infty.		 
 \label{eq:conv_main2}
\end{equation}
\end{thm}
We call the limiting Radon measure $\nu_q$ 
the $q$th persistence diagram of a stationary ergodic point process
$\Phi$. In non-ergodic case, by using the ergodic decomposition (cf. \cite{gs}), the right-hand side in \eqref{eq:conv_main2} is replaced by the random measure 
$\nu_{q, \omega}$ which is measurable with respect to the translation invariant $\sigma$-field 
$\cI$ defined in Section~\ref{sec:rm}.

\begin{rem}
 The set $\Delta$ is topologically the same as the triangle 
 \[
 	\{(x,y) \in \bR^2 : 0 \le x < y \le 1\}
\] 
with open boundary
 $\partial \Delta = \{(x,x) \in \bR^2 : 0 \le x \le 1\}$.
 Although we do not consider the mass on $\partial \Delta$, intuitively
 speaking, the (virtual) mass on $\partial \Delta$ comes from
 configurations of special forms such as three vertices of a right triangle. 
 With the vague convergence, we do not see the mass escaping towards the
 boundary $\partial \Delta$ in the limit $L \to \infty$. 
 In applications, the mass appearing
 near the boundary is often considered to be fragile under perturbation.  
\end{rem}

The limiting measure $\nu_q$ may be trivial. Indeed, for
 \cech complexes, $\nu_q = 0$ for $q \ge N$. This is just because
 there is no configuration in $\bR^{N}$ that realizes $q$th homology
 class for $q \ge N$. In order to characterize the support of $\nu_q$,
we introduce the notion of realizability of a point in a persistence diagram. 

\begin{defn}
We say that a point $(b,d) \in \Delta$ is \textit{realizable by} $\Xi \in
  \finite(\bR^{N})$ in the $q$th persistent homology
  if $(b,d)$ is contained in the $q$th persistence diagram of the
$\kappa$-filtration over $\fp$, i.e., $\xi_q(\fp)(\{(b,d)\}) \ge 1$. 
If such $\fp$ exists for $(b,d)$, we call $(b,d)$ a \textit{realizable} point.
  We denote by $R_q = R_q(\kappa)$
  the set of all realizable points in the $q$th persistent
 homology of the $\kappa$-filtration. 
\end{defn}


  \begin{ex}\label{rem:realizable}
   For $\alpha > 0$ and $\sigma \in \finite(\bR^{N})$,
   we define $\alpha \sigma \in \finite(\bR^{N})$ by
   $\alpha \sigma = \{\alpha x \in \bR^{N} : x \in \sigma\}$. 
It is easy to see that if $\kappa$ is homogeneous in the sense that
  $\kappa(\alpha \sigma) = \alpha \kappa(\sigma)$ for every
  $\alpha > 0$ and $\sigma \in \finite(\bR^{N})$,
  then $R_q(\kappa)$ forms a cone in $\Delta$.
Since both $\kappa_C$ and $\kappa_R$ given in Example~\ref{ex:kappas}
  are homogeneous, we can see that $R_q(\kappa_C)$ and $R_q(\kappa_R)$ are cones for every $q \ge 0$. 
In particular, for \cech complexes, we have 
 \begin{align}\label{eq:realizable_cech}
 R_q(\kappa_C) =
  \begin{cases}
  \{0\} \times (0, \infty], & \text{if $q=0$}, \\
   \{(b,d): 0<b < d <\infty\}, & \text{if $q=1,2,\dots,N-1$}, \\
   \emptyset, & \text{if $q=N, N+1,\dots$}. 
  \end{cases}
 \end{align}
 \end{ex}

 It is clear that $\support \nu_q \subset \overline{R_q(\kappa)}$.
 Indeed, if $x \not\in \overline{R_q(\kappa)}$, there exists $\eps>0$ such
 that $\xi_{q,L}(B_{\eps}(x)) = 0$, where $B_{\eps}(x)$ is the open
 $\eps$-neighborhood of $x$. It follows from the vague convergence
 \eqref{eq:conv_main1} that $\nu_q(B_{\eps}(x)) = 0$. Therefore $x
 \not\in \support \nu_q$.
 In Theorem~\ref{thm:support}, we give sufficient conditions for a point
in $R_q(\kappa)$ to be in the support of $\nu_q$.
The following result, as a consequence of that general theorem, states that $\support \nu_q$ coincides
with $\overline{R_q(\kappa)}$ under conditions that $\kappa$ is Lipschitz
 continuous and all local densities of the point process $\Phi$ are almost surely
 positive with respect to the Lebesgue measures. 

\begin{thm}\label{thm:positivity0}
 Let $\Phi$ be a stationary point process on $\bR^{N}$ and
$\Theta$ its probability distribution. 
Assume that for every compact set $\La \subset \bR^{N}$, 
 the restriction $\Theta|_{\La}$ on $\La$
 is absolutely continuous with respect to
$\Pi|_{\La}$ and the Radon-Nikodym density $d\Theta|_{\La}/d\Pi|_{\La}$ is
 strictly positive $\Pi|_{\La}$-almost surely,
 where $\Pi$ is the distribution of a homogeneous Poisson point process on $\bR^{N}$.
 In addition, assume that $\kappa$ on $\finite(\bR^{N})$ is Lipschitz 
 continuous with respect to the Hausdorff distance.
 Then, $\support \nu_q = \overline{R_q(\kappa)}$.
\end{thm}

\begin{ex}
 Homogeneous Poisson point processes, Gibbs measures 
 and Ginibre point processes satisfy the assumption in Theorem~\ref{thm:positivity0}.
Thus if $\kappa$ is Lipschitz continuous with respect to the Hausdorff
 distance, then $\support \nu_q = \overline{R_q(\kappa)}$ for such point
 processes. In particular, for \cech filtrations,
 $\support \nu_q = {\Delta}, (q=1,\dots, N-1)$. 
 On the other hand, the shifted lattice considered in
 Example~\ref{ex:shifted_lattice} does not satisfy the assumption and 
 $\support \nu_q$ for the \cech filtration turns out to be a singleton. 
\end{ex}

For the proof of Theorem~\ref{thm:LLN-PD},
we exploit a general theory of Radon
measures for the vague convergence (cf.~\cite{billingsley,
Kallenberg}). 
In particular, we show that the convergence of the values of measures on
the class $\{A_{r, s} = [0,r] \times (s, \infty] : 0 \leq r \le s < 
\infty\}$ is enough to ensure the vague convergence of random measures in Theorem
\ref{thm:LLN-PD}. The value of $\xi_{q,L}$ on $A_{r,s}$ is nothing but
the persistent Betti number  
\[
\beta_q^{r,s}(\K(\Phi_{\Lambda_L})) = \xi_{q,L} ([0,r] \times (s,
\infty]) = |\{(b_i, d_i) : 0 \le b_i \le r \le s< d_i \}|. 
\]
Here $|A|$ denotes the cardinality of a finite set $A$. Later, $|A|$ is also used to denote the Lebesgue measure of a set $A \in \bR^{N}$. The meaning is clear from the context.  
Hence Theorem \ref{thm:LLN-PD} follows from the following strong law
of large numbers for persistent Betti numbers. 
\begin{thm}\label{thm:LLN-Betti}
Assume that $\Phi$ is a stationary point process having all finite moments. Then  for any $0 \le r \le s < \infty$, there exists a constant $\hat \beta^{r,s}_q$ such that 
\[
	\frac{\E[\beta_q^{r,s} (\K(\Phi_{\Lambda_L}))] }{L^{N}} \to \hat \beta^{r,s}_q \text{ as } L \to \infty.
\] 
In addition, if $\Phi$ is ergodic, then 
\[
\frac{\beta_q^{r,s} (\K(\Phi_{\Lambda_L})) }{L^{N}} \to \hat \beta^{r,s}_q \text{ almost surely as } L \to \infty.
\] 
\end{thm}

Note that, for $r = s$, the persistent Betti number becomes the usual
Betti number, i.e., $\beta_q^{r,r} (\K(\Phi_{\Lambda_L})) =
\beta_q(K(\Phi_{\La_L}, r))$. Hence, this result is a generalization of
Lemma 3.3 and Theorem 3.5 in \cite{ysa}. 

For Poisson point processes, we also generalize the central limit theorem in \cite{ysa} for Betti numbers to persistent Betti numbers as follows.

\begin{thm}\label{thm:CLT}
	Let $\Phi$ be a homogeneous Poisson point process on
 $\bR^{N}$ with
 unit intensity. Then for any $0 \le r \le s < \infty$, there exists a
 constant $\sigma_{r,s}^2 = \sigma_{r,s}^2 (q)$ such that  
	\[
		\frac{\beta_q^{r,s} (\K(\Phi_{\Lambda_L})) -
 \E[\beta_q^{r,s} (\K(\Phi_{\Lambda_L})) ]}{L^{N/2}} \dto \cN(0,
 \sigma_{r,s}^2)\text{ as } L \to \infty. 
	\]
\end{thm}

We remark that the proof of the central limit theorem for (usual) Betti numbers in \cite{ysa} uses the Mayer-Vietoris
exact sequence to estimate the effect of one point adding on the Betti
number. However, in the setting of persistent homology, although we can
obtain the Mayer-Vietoris exact sequence for each parameter $r$, we do
not have the exactness property with regard to the parameter
change. Hence, the same technique may not be applicable to the case of
persistent Betti numbers. Instead, we give an alternative (and
elementary) proof for the generalization. Remark also that by
establishing the strong stabilization, the central limit theorem for
Betti numbers of \cech complexes built over binomial point processes is
also established in \cite{ysa}. In this case, the positivity of the
limiting variance is also proved under a certain condition on radius parameter
$r$. The positivity problem for the limiting variance is left open in 
case of persistent Betti numbers of general $\kappa$-complexes built
over homogeneous Poisson point processes.

The organization of this paper is given as follows. Necessary
concepts and properties of persistent homology and random measures are
explained in Section \ref{sec:geom_ph} and Section \ref{sec:rm},
respectively.
Theorem~\ref{thm:support} which characterizes the support of limiting
persistence diagrams is stated and proved in
Section~\ref{sec:positivity}.  
The proofs of Theorems \ref{thm:LLN-PD}, \ref{thm:positivity0}, \ref{thm:LLN-Betti} and
\ref{thm:CLT} are given in Sections \ref{sec:LLN-PD},
\ref{sec:positivity}, \ref{sec:LLN-Betti}  and \ref{sec:CLT-pbetti} in order. In Section
\ref{sec:conclusion}, we summarize the conclusions of the paper and show
some future works.

\section{Geometric models and persistent homology}\label{sec:geom_ph}
In this section, we assume fundamental properties about simplicial
complexes and their homology. 
For details, the reader may refer to Appendix \ref{ap:topology} or \cite{eh,hatcher}. 

\subsection{Geometric models for point processes}\label{sec:gm}
Let $\kappa\colon\finite(\bR^{N})\rightarrow[0,\infty]$ be a function
satisfying the three conditions explained in Section
\ref{sec:intro}, where $\finite(\bR^{N})$ is the collection of all finite
subsets in $\bR^{N}$. For such a function $\kappa$, the $\kappa$-filtration $\K(\Phi)=\{K(\Phi,t)\}_{t\geq 0}$
can be defined in the same way as in \eqref{eq:sc_model} for an infinite
point configuration (or a point process) $\Phi \subset \bR^{N}$ as
well as for a finite point process.

We remark that all vertices (i.e., $0$-simplices) exist at $t=0$.
Also, all simplices in $K(\Phi, t)$ possessing a point
$x$ must lie in the ball $\bar B_{\rho(t)}(x)$
since $\{x, x_1,\dots,x_q\} \in K(\Phi,t)$ with Assumption
(K3) implies that 
$\|x - x_i\| \le \rho(\kappa(x, x_i)) \le \rho(t)$ for all $i$. 
Here $\bar B_r(x)=\{y\in\bR^{N} : \|y-x\|\leq r\}$ is the closure of 
$B_r(x)$ which denotes the open ball of radius $r$ centered at $x$. 
Hence, for each parameter $t$, the presence of simplices containing $x$
is localized in $\bar B_{\rho(t)}(x)$.

This geometric model includes some of the standard models studied in
random topology. For instance, the {\v C}ech complex $C(\Phi,t)$ is a
simplicial complex with the vertex set $\Phi$ and, for each parameter
$t$, it is defined by 
\[
	\sigma=\{x_0,\dots,x_q\}\in C(\Phi,t) \Longleftrightarrow \bigcap_{i=0}^q \bar B_t(x_i)\neq \emptyset
\]
for $q$-simplices. 
Similarly, the Rips complex $R(\Phi,t)$ with a parameter $t$ is defined
by
\[
	\sigma=\{x_0,\dots,x_q\}\in R(\Phi,t) \Longleftrightarrow \bar B_t(x_i)\cap \bar B_t(x_j)\neq \emptyset \text{ for } 0\leq  i< j\leq q.
\]
It is clear that these geometric models are generated by the functions
given in Example~\ref{ex:kappas}. 
We note that $R(\Phi,t/2) \subset C(\Phi, t) \subset R(\Phi, t)$
since $\kappa_R \le \kappa_C \le 2\kappa_R$. 


In Example \ref{rem:realizable}, we showed the set $R_q(\kappa_C)$
of the realizable points in (\ref{eq:realizable_cech}) for the \cech
filtration. We are going to give a proof of that fact here. The cases $q=0$ and $q\geq N$ are easily derived. For $q=1,\dots, N-1$, we will show that any birth-death pair $(b,d)$ with $0<b<d<\infty$ is realizable by explicitly
constructing the points $\Xi\in\finite(\bR^{N})$ realizing $(b,d)$. Indeed, let $S^q_{d}\subset \bR^{N}$  be a $q$-dimensional sphere with radius $d$ and $z_0$ be any point in $S^q_d$. We choose points $\Xi_+$ on $S_d^q \cap  \partial \bar B_b(z_0)$ and $\Xi_-$ on $S_d^q \setminus \bar B_b(z_0)$ such that 
\begin{itemize}
	\item[(i)] $S_d^q \cap \bar B_b(z_0) \subset \bigcup_{x \in \Xi_+} \bar B_b(x) $;
	\item[(ii)] $\bigcup_{x \in \Xi_-} \bar B_r(x)$ covers $S_d^q \setminus \bar B_b(z_0)$ earlier than $r=b$;
	\item[(iii)] $\bigcup_{x\in\Xi} \bar B_r(x)$
provides the generator of $q$-dimensional homology corresponding to
$S^q_{d}$, where $\Xi =\Xi_+\sqcup\Xi_-$.
\end{itemize}
Then the birth-death pair of the generator $\bigcup_{x\in\Xi} \bar B_r(x)$ is exactly $(b,d)$.


\subsection{Persistent homology}

Let $\K=\{K_r : r\geq 0\}$ be a (right continuous) filtration of
simplicial complexes, i.e., $K_r\subset K_s$ for $r\leq s$ and
$K_r=\bigcap_{r<s}K_s$. In this paper, the homology $H_q(K)$ of a simplicial complex $K$ is defined on an arbitrary field $\bF$.
For $r\leq s$, we denote the linear map on homologies induced
from the inclusion $K_r\hookrightarrow K_s$ by $\iota^s_r \colon H_q(K_r)\rightarrow
H_q(K_s)$. The {\it $q$th persistent homology}
$H_q(\K)=(H_q(K_r),\iota^s_r)$ of $\K$ is defined by the family of homologies $\{H_q(K_r): r\geq 0\}$ and the induced linear maps
$\iota^s_r$ for all $r\leq s$. 

A {\em homological critical value} of $H_q(\K)$ is a number $r> 0$ 
such that the linear map $\iota^{r+\eps}_{r-\eps} \colon
H_q(K_{r-\eps})\rightarrow H_q(K_{r+\eps})$ is not 
isomorphic for any sufficiently small $\eps>0$. The persistent
homology $H_q(\K)$ is said to be {\em tame} if $\dim H_q(K_r)<\infty$
for any $r\geq 0$ and the number of homological critical values is finite.  
A tame persistent homology $H_q(\K)$ has a nice decomposition property:
\begin{thm}[\cite{zc}]\label{thm:decomposition}
Assume that $H_q(\K)$ is a tame persistent homology. Then, there uniquely 
exist 
 indices $p\in\bZ_{\geq 0}$ and
 $b_i,d_i\in \Rbar_{\geq 0}=\bR_{\geq 0}\sqcup\{\infty\}$ with $b_i<d_i, i = 1,2,\dots, p,$ such that the following isomorphism holds: 
\begin{align}\label{eq:decom}
H_q(\K)\simeq\bigoplus_{i=1}^p I(b_i,d_i).
\end{align}
Here, $I(b_i,d_i)=(U_r,f^s_r)$ consists of a family of vector spaces
\begin{align*}
U_r=
\begin{cases}
\bF,&b_i\leq r < d_i,\\
0,&{\rm otherwise},
\end{cases}
\end{align*}
and the identity map $f^s_r=\id_{\bF}$ for $b_i\leq r \leq s<d_i$.
\end{thm}

Each summand $I(b_i,d_i)$ in (\ref{eq:decom}) is called a generator of the persistent homology and $(b_i,d_i)$ is called its birth-death pair. From the unique decomposition in Theorem \ref{thm:decomposition}, we define the $q$th persistence diagram as a multiset in  $\Rbar^2_{\geq 0}$,
\begin{align}\label{eq:pd}
D_q(\K)=\{(b_i,d_i)\in \Rbar_{\geq 0}^2 : i=1,\dots,p\}.
\end{align}
By denoting the
multiplicity of the point $(b,d)$ in (\ref{eq:pd}) by $m_{b,d}\in \bN_0=\{0,1,2,\dots\}$, we can also express
the decomposition  (\ref{eq:decom}) as 
\begin{align*}
H_q(\K)\simeq\bigoplus_{(b,d)} I(b,d)^{m_{b,d}}.
\end{align*}
Later, we identify a persistence diagram $D_q(\K)$
as an integer-valued Radon measure $\xi = \sum_{(b,d)} m_{b,d}
\delta_{(b,d)}$ rather than as a multiset. 

Although our target object $\K(\Phi)$ is built over infinite points, 
all persistent homologies studied in this paper are defined on the
geometric models with finite points. Hence, the persistent homology
becomes tame, and the persistence diagrams are well-defined. 

\begin{ex}
 \begin{figure}
\begin{minipage}{5.5cm}
\includegraphics[scale=0.6]{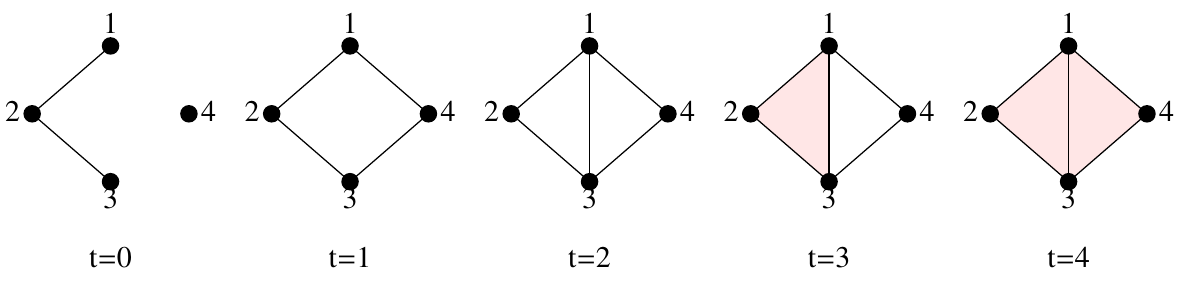}
\end{minipage}
 \hfill
\begin{minipage}{4cm}
\includegraphics[scale=0.3]{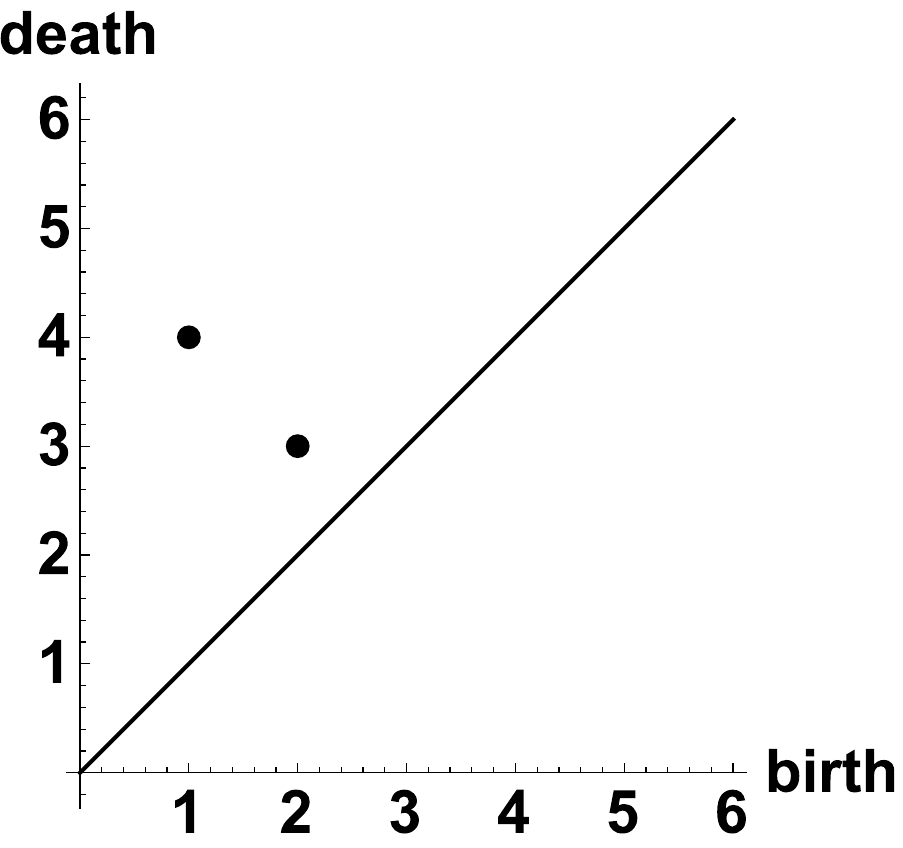}
\end{minipage}
\caption{A filtration of simplicial complexes and the $1$st persistence diagram} 
\label{fig:filtration_PD}
 \end{figure}
 In Figure~\ref{fig:filtration_PD}, 
 two (1-dim)cycles appear at times $1$ and $2$ and disappear at times 
 $3$ and $4$. The representation corresponding to $H_1(\K)$ is given as 
\begin{align*}
 0 \to \bF(c_1 + c_2) \to  \bF(c_1)& \oplus \bF(c_2)  \to \bF(c_1) \oplus \bF(c_2) / \bF(c_1)  \\
 &\to
 \bF(c_1) \oplus \bF(c_2) / \bF(c_1) \oplus \bF(c_2) \simeq 0, 
\end{align*}
 where $c_1 = \bra 12\ket + \bra 23 \ket+ \bra 31 \ket$
 and $c_2 = \bra 13 \ket + \bra 34 \ket + \bra 41 \ket$ and each arrow is
 the linear map induced by inclusion. 
 As pairs of birth-death times, we have $(1,4)$ and $(2,3)$ since 
the decomposition of the representation is given by 
\begin{align*}
H_1(\K) &= (0 \to \bF(c_1 + c_2) \to \bF(c_1 + c_2) \to \bF(c_1
 + c_2) \to 0) \\   
&\quad \oplus (0 \to 0 \to \bF(c_1) \to 0 \to 0).   
\end{align*}
\end{ex}

\begin{rem}
More generally, a persistence module $\U=(U_a,f_a^b)$ on
 $\Rbar_{\geq 0}$ is defined by a sequence of general vector
 spaces $U_a, a\geq 0,$ and linear maps $f_a^b \colon U_a\rightarrow U_b$ for
 $a\leq b$ satisfying $f_a^c=f_b^c\circ f_a^b$. Under the same
 definition of the tameness, we can similarly define its persistence
 diagrams.  
\end{rem}

 \begin{rem}
  There is another definition of persistent homology as graded modules
  over a monoid ring for the continuous parameter (resp. a polynomial ring for the
  discrete parameter). See, for example, \cite{hs_frieze}. 
 \end{rem}
 
  \begin{rem}
  The persistent homology $H_q(\K)$ defined over the whole $\Phi$ is not
   tame in general while $H_q(\K_L)$ defined over a restriction $\Phi_{\La_L}$ is tame. 
  Theorem~\ref{thm:LLN-PD} informally says that 
  \[
   \frac{1}{L^{N}} H_q(\K_{L}) \stackrel{\simeq}{\to} H_q(\K) =
   \int^{\oplus}_{\Delta} I(x,y)
  \nu_q(dxdy),
  \]
  where $\K_L = \{K(\Phi_{\La_L}, t)\}_{t \ge 0}$, where
   $\int_{\Delta}^{\oplus}$ denotes the direct integral of interval
   representations (cf. \cite{rs}). 
  \end{rem}

    \begin{rem}\label{rem:curve} 
     In our paper, we use the persistence diagram for representing
     topological information obtained from filtrations. 
     People sometimes use the so-called barcode representation in which
     each persistence interval $I(b,d)$ is represented as
     a barcode $[b,d]$ (cf. \cite{ya}).   
    We consider the marginal measure of persistence diagram 
     on death times (also on birth times),
     i.e., the induced measure $\xi^{(death)}$ obtained from
    a measure $\xi$ on $\Delta$ by the projection $\Delta \ni (x,y)
     \mapsto y \in [0,\infty)$.
    The marginal measure $\xi_{q,L}^{(death)}$
    of a persistence diagram $\xi_{q,L}$ 
     induces a (scaled) right-continuous step function
     $f_{q,L}(t) = L^{-N} \xi_{q,L}^{(death)}([0,t])$, which
     corresponds to the one obtained by simulation in \cite{ya}.  
     The function $f_{q, L}(t)$ is also expected to converge to a limit
     $f_{q,\infty}(t)$ as $L \to \infty$, however, it does not
     necessarily coincide with
     $f_q(t) := \nu_q^{(death)}([0,t])$ because of the mass escaping
     to $\partial \Delta$.   
    \end{rem}

\subsection{Persistent Betti numbers}

For a filtration $\K$, the $(r,s)$-persistent Betti number \cite{elz} is defined by
\begin{align}\label{eq:rsbetti}
\beta^{r,s}_q(\K)=\rank\frac{Z_q(K_r)}{Z_q(K_r)\cap B_q(K_s)} \quad (r \le s), 
\end{align}
where $Z_q(K_r)$ and $B_q(K_r)$ are the $q$th cycle group and boundary
group, respectively.
We remark that this is equal to the rank of $\iota_r^s \colon H_q(K_r)\rightarrow H_q(K_s)$, because
\[
	\image \iota_r^s\simeq\frac{\frac{Z_q(K_r)}{B_q(K_r)}}{\frac{Z_q(K_r)\cap B_q(K_s)}{B_q(K_r)}}\simeq \frac{Z_q(K_r)}{Z_q(K_r)\cap B_q(K_s)}. 
\]
Thus, from the decomposition of the persistent homology, we have 
\begin{align*}
\beta^{r,s}_q(\K)=\sum_{b\leq r, d> s}m_{b,d}.
\end{align*}
This means that the $(r,s)$-persistent Betti number $\beta_q^{r,s}(\K)$ counts the number of birth-death pairs in the persistence diagram $D_q(\K)$ located in the gray region of Figure $\ref{fig:pd}$. 
\begin{figure}[htbp]
 \begin{center}
  \includegraphics[width=0.3\hsize]{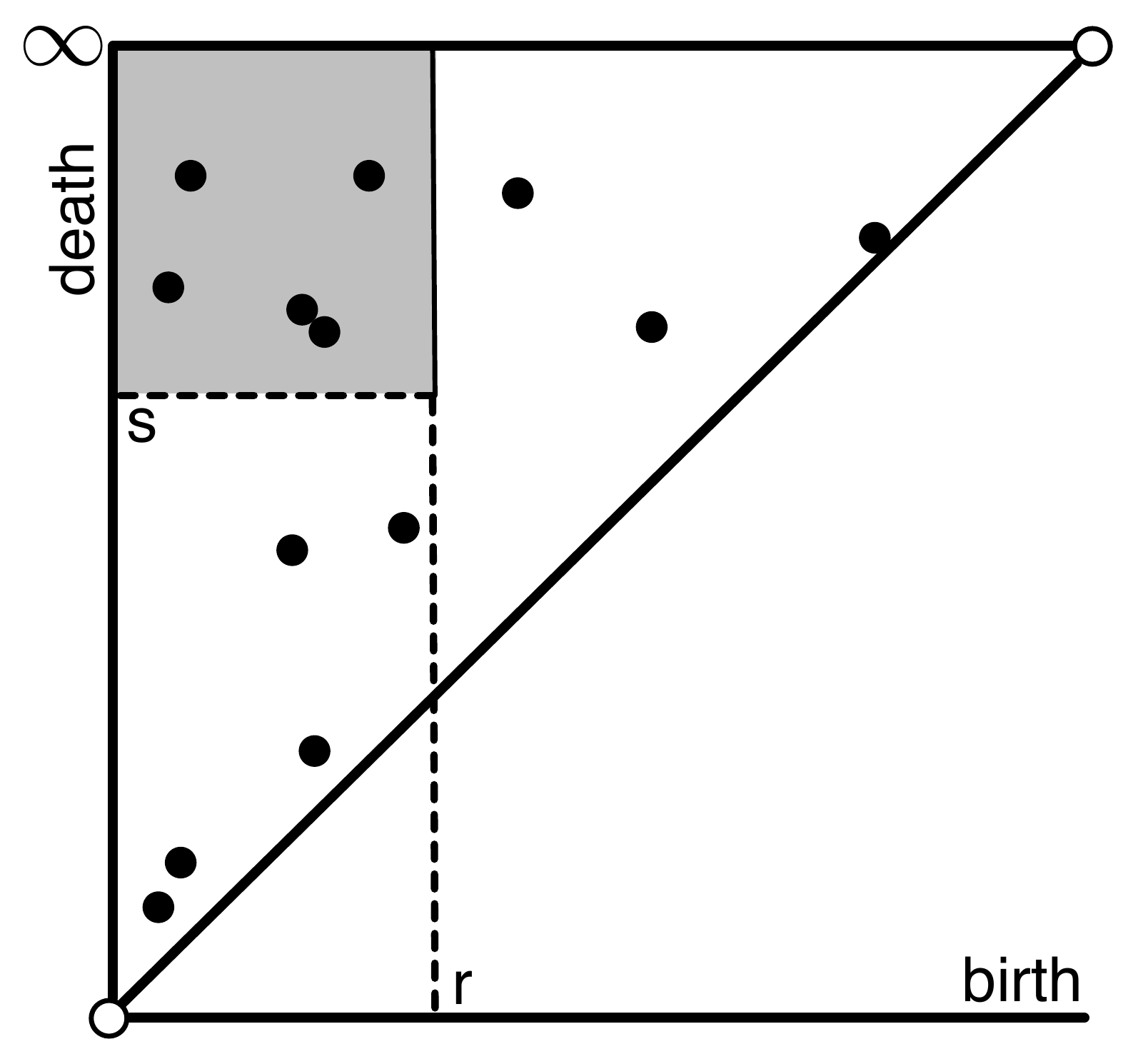}
  \caption{$\beta^{r,s}_q(\K)$ counts the number of generators in the
  gray region.}
  \label{fig:pd}
 \end{center}
\end{figure}

\begin{lem}\label{lem:truncation}
Let $\U=(U_a,f_a^b)$ be a persistence module on $\Rbar_{\geq 0}$ and let $\V=(V_a,g_a^b)$ be its truncation on the interval $[r,s]$, meaning that
\begin{align}
V_a=
\left\{\begin{array}{ll}
U_r, & a\leq r,\\
U_a, & r\leq a\leq s,\\
U_s, & a\geq s,
\end{array}\right. \quad
g_a^b=\left\{\begin{array}{ll}
f_a^b, & r\leq a\leq b\leq s,\\
f_a^s, & r\leq a\leq s\leq b,\\
f_r^b, & a\leq r\leq b\leq s,\\
f_r^s, & a\leq r\leq s\leq b.\\
\end{array}\right.
\end{align}
For interval decompositions $\U\simeq\oplus I(b,d)^{m_{b,d}}$ and $\V\simeq\oplus I(b,d)^{n_{b,d}}$, let 
\begin{align*}
\beta^{r,s}(\U)=\sum_{b\leq r, d>s}m_{b,d},\quad
\beta^{0,\infty}(\V)=\sum_{b=0, d=\infty}n_{b,d}.
\end{align*}
Then $\beta^{r,s}(\U)=\beta^{0,\infty}(\V)$.
\end{lem}

\begin{proof}
This is because $\beta^{r,s}(\U)=\rank f_r^s=\beta^{0,\infty}(\V)$.
\end{proof}

Here, we recall the following basic facts in linear algebra for later use.
\begin{lem}\label{lem:rank_lemma}
Let $A,B,U,V$ be subspaces of a vector space satisfying $A\subset U$ and $B\subset V$. Then,
\begin{align*}
\rank \frac{U\cap V}{A\cap B}\leq \rank \frac{U}{A}+\rank\frac{V}{B}.
\end{align*}
\end{lem}
\begin{proof}
Let us consider a surjective map
\begin{align*}
f \colon \frac{U\cap V}{A\cap B}\longrightarrow \frac{U\cap V}{A\cap V},\quad [c]\longmapsto [c]. 
\end{align*}
Since the kernel of the map is $\displaystyle{\kernel f =\frac{A\cap V}{A\cap B}}$, we have
\begin{align*}
\cfrac{~~\cfrac{U\cap V}{A\cap B}~~}{~~\cfrac{A\cap V}{A\cap B}~~}\simeq \frac{U\cap V}{A\cap V}.
\end{align*}
This leads to the conclusion that
\begin{align*}
\rank\frac{U\cap V}{A\cap B}=\rank \frac{A\cap V}{A\cap B} + \rank\frac{U\cap V}{A\cap V}
\leq \rank \frac{V}{B}+ \rank\frac{U}{A}.
\end{align*}
Here, the last inequality has been derived from the injectivity of the following maps
\begin{align*}
&\phi \colon  \frac{A\cap V}{A\cap B}\longrightarrow \frac{V}{B},\quad [c]\longmapsto [c],\\
&\psi \colon \frac{U\cap V}{A\cap V}\longrightarrow \frac{U}{A},\quad [c]\longmapsto [c].
\end{align*}
The proof is complete.
\end{proof}

\begin{lem}\label{lem:col_add}
Let 
$D=[A ~B]$ be a matrix composed by submatrices $A$ and $B$. Let $\ell$ be the number of columns in $B$. 
Then,
\begin{align*}
\rank D\leq \rank A +\ell,
\quad
\dim\kernel D\leq\dim\kernel A+\ell.
\end{align*}
\end{lem}

\begin{proof}
Let $B=[b_1\cdots b_\ell]$, where $b_i$ is the $i$th column vector of $B$, and set $D^{(i)}=[A~b_1\cdots b_i]$. Then, for each $i$, we have
\begin{align*}
\rank D^{(i)}\leq \rank D^{(i-1)} +1,
\quad
\dim\kernel D^{(i)}\leq\dim\kernel D^{(i-1)}+1.
\end{align*}
Hence, in total, we have the desired inequalities.
\end{proof}

Now, we show a basic estimate on the persistent Betti number for nested filtrations $\K\subset \tilde\K$. First, we note the following property.
\begin{lem}\label{lem:one_add}
Let $\K$ be a filtration. For a fixed $a>0$, let $\tilde\K=\{\tilde K_t : t\geq 0\}$ be a filtration given by 
\begin{align*}
\tilde K_t=\left\{\begin{array}{ll}
K_t, & t< a,\\
K_t\cup\sigma, & t\geq a,
\end{array}\right.
\end{align*}
where $\sigma$ is a new simplex added on $K_a$. 
Then, $\beta_q^{r,s}(\tilde \K)=\beta_q^{r,s}(\K)$ for $\dim\sigma\neq q,q+1$. 
For $\dim\sigma=q,q+1$, 
\begin{align*}
|\beta_q^{r,s}(\tilde \K)-\beta_q^{r,s}(\K)|\leq 
\left\{
\begin{array}{ll}
0, & \tilde{K}_r=K_r~{\rm and}~\tilde{K}_s=K_s,\\
1, & {\rm otherwise.}
\end{array}
\right.
\end{align*}
\end{lem}

\begin{proof}

We first note that 
\begin{align*}
\beta_q^{r,s}(\tilde\K)-\beta_q^{r,s}(\K)=&\rank Z_q(\tilde K_r)-\rank Z_q(\tilde K_r)\cap B_q(\tilde K_s)\\
&-(\rank Z_q(K_r)-\rank Z_q(K_r)\cap B_q(K_s))\\
=&\rank\frac{Z_q(\tilde K_r)}{Z_q(K_r)}-\rank\frac{Z_q(\tilde K_r)\cap B_q(\tilde K_s)}{Z_q(K_r)\cap B_q(K_s)}.
\end{align*}
Hence, the statement is trivial for $\dim\sigma\neq q,q+1$. 
Furthermore, when $\tilde{K}_r=K_r$ and $\tilde{K}_s=K_s$, 
we also have $\beta_q^{r,s}(\tilde \K)=\beta_q^{r,s}(\K)$.

Let $\dim\sigma=q$.  
Then, it follows from Lemma \ref{lem:col_add} that
\begin{align*}
&\rank\frac{Z_q(\tilde K_r)}{Z_q(K_r)}=\rank Z_q(\tilde K_r)-\rank Z_q(K_r)=0{\rm~or~}1,\\
&\rank\frac{B_q(\tilde K_s)}{B_q(K_s)}=0.
\end{align*}
Also, from Lemma \ref{lem:rank_lemma}, we have
\begin{align*}
0\leq \rank\frac{Z_q(\tilde K_r)\cap B_q(\tilde K_s)}{Z_q(K_r)\cap B_q(K_s)}\leq
\rank\frac{Z_q(\tilde K_r)}{Z_q(K_r)}+\rank\frac{B_q(\tilde K_s)}{B_q(K_s)}\leq\rank\frac{Z_q(\tilde K_r)}{Z_q(K_r)}.
\end{align*}
Therefore, $|\beta_q^{r,s}(\tilde\K)-\beta_q^{r,s}(\K)|\leq 1$. The statement for $\dim\sigma=q+1$ is similarly proved. 
\end{proof}

\begin{lem}\label{lem:multi_add}
Let $\K=\{K_t\}_{t\geq 0}$ and $\tilde \K=\{\tilde K_t\}_{t\geq 0}$ be filtrations with $K_t\subset \tilde K_t$ for $t\geq 0$. 
Then, 
\begin{align*}
|\beta_q^{r,s}(\tilde \K)-\beta_q^{r,s}(\K)|\leq
 \sum_{j=q,q+1}\left(|\tilde K_{s,j}\setminus
 K_{s,j}|+|\{\sigma\in K_{s,j}\setminus K_{r,j}:\tilde t_\sigma\leq r
 \}|\right), 
\end{align*}
where $\tilde K_{t,j}~(or~K_{t,j})$ is the set of $j$-simplices in
 $\tilde K_t~(or~K_t)$, and $\tilde t_\sigma~(or~t_\sigma)$ is the birth
 time of $\sigma$ in the filtration $\tilde\K~(or~\K)$. 

\end{lem}

\begin{proof}
We first decompose $\tilde K_s\setminus K_r=Y\sqcup Y^c$ by
\begin{align*}
Y=(\tilde K_s\setminus K_s)\sqcup\{\sigma\in K_{s}\setminus K_{r}:\tilde
 t_\sigma\leq r\}, \ 
Y^c=\{\sigma\in K_{s}\setminus K_{r}:r<\tilde t_\sigma\leq t_{\sigma}\}. 
\end{align*}
We use the same notation $Y_j$ for the set of $j$-simplices in $Y$. 
For the simplices in $\tilde K_s\setminus K_r=\{\sigma_i\}_{i=1}^L$, we assign
 their indices so that the birth times are in increasing order $\tilde t_{\sigma_1}\leq\dots\leq\tilde t_{\sigma_L}$ and $K_r\cup \{\sigma_1,\dots,\sigma_\ell\}$ becomes a simplicial complex for each $\ell$. 
We note that $\tilde t_{\sigma}\leq t_\sigma$. Furthermore, it suffices to consider the truncations of $\K$ and $\tilde\K$ on $[r,s]$ from Lemma \ref{lem:truncation}.

Now, we inductively construct a sequence of filtrations $\K=\K^0\subset
 \K^1\subset\dots\subset \K^L=\tilde \K$.  The filtration
 $\K^i=\{K^i_t:t\geq 0\}$ is given by adding a simplex $\sigma_i$ to
 $\K^{i-1}$ at $\tilde t_{\sigma_i}$, i.e., 
\begin{align*}
K^i_t=\left\{
\begin{array}{ll}
K_t^{i-1},&t<\tilde t_{\sigma_i},\\
K^{i-1}_t\cup\{\sigma_i\},& t\geq\tilde t_{\sigma_i}.
\end{array}
\right.
\end{align*}
Then, it follows from Lemma \ref{lem:one_add} that
 $|\beta^{r,s}_q(\K^i)-\beta^{r,s}_q(\K^{i-1})|\leq1$ for $\sigma_i\in
 Y$, since $K^i_r\neq K^{i-1}_r$ or $K^i_s\neq K^{i-1}_s$ holds. 
On the other hand, Lemma \ref{lem:one_add} implies $\beta^{r,s}_q(\K^i)=\beta^{r,s}_q(\K^{i-1})$ for $\sigma\in Y^c$.
Therefore, 
\begin{align*}
|\beta_q^{r,s}(\tilde \K)-\beta_q^{r,s}(\K)|\leq
 \sum_{i=1}^L|\beta_q^{r,s}(\K^i)-\beta_q^{r,s}(\K^{i-1})|
\le |Y_q|+|Y_{q+1}|,
\end{align*}
which completes the proof of Lemma~\ref{lem:multi_add}.
\end{proof}

\begin{rem}{\rm
Let $\Phi, \tilde{\Phi} \in \finite(\bR^{N})$ with $\Phi \subset \tilde{\Phi}$, and $t_\sigma$ and $\tilde{t}_\sigma$ be the birth times of the simplex $\sigma$ in the $\kappa$-filtrations $\K(\Phi)$ and $\K(\tilde{\Phi})$, respectively. Then, it is obvious that 
$\tilde{t}_{\sigma} = t_{\sigma}$ 
if $\sigma \subset \Phi \subset \tilde{\Phi}$.
Hence, for the estimate $|\beta^{r,s}_q(\K(\tilde\Phi))-\beta^{r,s}_q(\K(\Phi))|$, 
the second term obtained in Lemma \ref{lem:multi_add} does not appear under this setting. 
}
\end{rem}

\section{General theory of random measures}\label{sec:rm}
Let $S$ be a locally compact Hausdorff space with
countable basis and $\cS$ be the Borel $\sigma$-algebra on $S$. 
It is well known that $S$ is a Polish space, i.e., a complete separable 
metrizable space.  If needed, we take a metric $\rho$ which makes $S$
complete and separable. 
We denote by $\bdd(S)$ the ring 
of all relatively compact sets in $\cS$. 
A measure $\mu$ on $(S, \cS)$ is said to be a \textit{Radon measure}  
if $\mu(B) < \infty$ for every $B \in \bdd(S)$. 
Let $\rM(S)$ be the set of all Radon measures on
$(S, \cS)$ and $\cM(S)$ be the $\sigma$-algebra  generated by the mappings 
$\rM(S) \ni \mu \mapsto \mu(B) \in [0,\infty)$ for every $B \in \bdd(S)$.  

We say that a sequence $\{\mu_n\}_{n \ge 1} \subset \rM(S)$ converges
to $\mu \in \rM(S)$ \textit{vaguely} (or \textit{in the vague topology})
if $\bra \mu_n, f \ket \to \bra \mu, f \ket$
for every continuous function $f$ with compact support, where $\bra \mu, f \ket = \int_S f(x)d\mu(x)$. 
In this case, we write $\mu_n \vto \mu$. The space $\rM(S)$ equipped with
the vague topology again becomes a Polish space and its Borel
$\sigma$-algebra coincides with $\cM(S)$.

We denote by $\rN(S)$ the subset in $\rM(S)$ of all integer-valued Radon measures on $S$.  
Each element in $\rN(S)$ can be expressed as a sum of
delta measures, i.e., $\mu = \sum_{i} \delta_{x_i} \in \rN(S)$.
We note that the set $\rN(S)$ is a closed subset of $\rM(S)$ in the vague
topology. 

An $\rM(S)$-valued (resp.~$\rN(S)$-valued) random variable
$\xi = \xi_{\omega}$  on a probability space $(\Omega, \cF, \PP)$
is called a \textit{random measure} (resp. \textit{point process}) on $S$. 
If $\lambda_1(A) := \E[\xi(A)] < \infty$ for all $A \in \bdd(S)$,
then $\lambda_1$ defines
a Radon measure and is referred to as  the mean measure, or the
intensity measure of $\xi$. Sometimes we denote it by $\E[\xi]$.


In this paper, two kinds of point processes will appear.
 One is point processes on $\bR^{N}$ as spatial point data
and the other is point processes on
 $\Delta = \{(x,y) \in \Rbar^2 : 0 \leq x < y \leq \infty\}$
 as persistence diagrams. 
The former will be denoted by the upper case letters like $\Phi$ 
and the latter by the lower case letters like $\xi$.

The point process $\Phi$ on $\bR^{N}$ is called {\it stationary}, if the distribution $\PP\Phi^{-1}$ is invariant under translations, i.e., $\PP\Phi^{-1}_x=\PP\Phi^{-1}$ for any $x\in \bR^{N}$, where $\Phi_x$ is the translated point process defined by $\Phi_x(B)=\Phi(B-x)$ for $B\in\bdd(\bR^{N})$. For $A\subset \rM(\bR^{N})$, let $A_x=\{\mu_x :\mu\in A\}$ be a set of translated measures defined by $\mu_x(B)=\mu(B-x)$.
Given a point process $\Phi$, let $\cI$ be the translation invariant
$\sigma$-field in $\rN(\bR^{N})$, i.e., the class of subsets $I\subset
\rN(\bR^{N})$ satisfying 
\[
\PP\Phi^{-1}((I\setminus I_x)\cup (I_x\setminus I))=0
\]
for all $x\in\bR^{N}$.
Then, $\Phi$ is called {\it ergodic} if $\cI$ is trivial, that is, for
every $I\in\cI$, $\PP\Phi^{-1}(I) \in \{0,1\}$.

From now on and until the end of this section, we fix a space $S$ and write $\bdd$ and $\rM$ for $\bdd(S)$ and $\rM(S)$, respectively. 
For a subset $A\subset S$, we denote by $\partial A$ and $A^{\circ}$ the boundary and interior of $A$, respectively. 
For a measure $\mu \in \rM$, let $\bdd_\mu :=\{ B \in \bdd: \mu(\partial B) = 0\}$ be the class of relatively compact continuity sets of $\mu$.

\begin{lem}[{\cite[15.7.2]{Kallenberg}}]\label{lem:equivalent-condition-vague-convergence}
Let $\mu, \mu_1,\mu_2, \ldots \in \rM$. Then the following statements
 are equivalent: 
\begin{itemize}
	\item[\rm (i)] $\mu_n \vto \mu$;
	\item[\rm (ii)] $\mu_n (B) \to \mu(B)$ for all $B \in
		     \bdd_{\mu}$;
	\item[\rm (iii)] $\limsup_{ n \to \infty } \mu_n (F) \le \mu(F)$
		     and $\liminf_{n \to \infty} \mu_n (G) \ge \mu(G)$ 
		     for all closed $F \in \bdd$ and open $G \in \bdd$. 
\end{itemize}
\end{lem}

\begin{lem}[{\cite[15.7.5]{Kallenberg}}]\label{lem:tightness}
A subset $C$ in $\rM$ is relatively compact in the vague topology iff 
\[
	\sup_{\mu \in C} \mu(B) < \infty \ \text{for every $B \in \bdd$}.
\]
\end{lem}

A class $\mfA \subset \bdd$ is called a
\textit{convergence-determining class} (for vague convergence)
if for every $\mu\in \rM$ and every sequence $\{\mu_n\}\subset \rM$, the condition  
\[
	\mu_n(A) \to \mu(A) \text{ for all $A \in \mfA \cap \bdd_\mu$}
\]
implies the vague convergence  $\mu_n \vto \mu$.
A class $\mfA_\mu \subset \mfB_\mu$ is called a convergence-determining class for $\mu$ if for any sequence $\{\mu_n\} \subset \rM$, the condition 
\[
	\mu_n(A) \to \mu(A) \text{ as } n \to \infty \text{ for all } A \in \mfA_\mu,
\]
implies that $\mu_n \vto \mu$. By definition, a class $\mfA$ is a convergence-determining class if and only if for any $\mu \in \rM$, $\mfA_\mu = \mfA \cap \mfB_\mu$ is a convergence-determining class for $\mu$. 

We say that a class $\mfC$ has the \textit{finite covering property} if 
any subset $B \in \bdd$ can be covered by a finite union of $\mfC$-sets.

\begin{lem} \label{lem:existence-vague-limit}
Let $\mfA$ be a convergence-determining class with finite covering property. 
Let $\{\mu_n\}$ be a sequence of measures in $\rM$. 
If $\mu_n(A)$ converges to a finite limit for any $A\in \mfA$, then there exists a measure $\mu$ to which
the sequence $\{\mu_n\}$ converges vaguely.
\end{lem}
\begin{proof}
For any relatively compact set $B \in \bdd$, we can find a finite cover $\{A_i\}_{i=1}^m \subset
 \mfA$ of $B$ so that 
\[
	\limsup_{n \to \infty} \mu_n (B) 
\le \limsup_{n \to \infty} \mu_n(\cup_{i = 1}^m A_i) 
 \le \lim_{n \to \infty} \sum_{i=1}^m \mu_n(A_i) < \infty.
\]
Therefore the sequence $\{\mu_n\}_{n \ge 1}$ is relatively compact by \lref{lem:tightness}, and hence, there is a subsequence $\{\mu_{n_k}\}$ and
 $\mu \in \rM$ such that $\mu_{n_k} \vto \mu$, i.e., 
$\mu_{n_k}(A) \to \mu(A)$ for every $A \in \bdd_{\mu}$. This together
 with the assumption implies 
that $\mu_n(A) \to \mu(A)$ for every $A \in \mfA \cap \bdd_{\mu}$. 
Consequently, $\mu_n$ converges to $\mu$ vaguely from the definition of 
convergence-determining class. The proof is complete.
\end{proof}

\begin{prop}\label{prop:LLN}
 Let $\mfA$ be a convergence-determining class 
with finite covering property and the property that for every $\mu \in \rM$, it contains a \textit{countable} convergence-determining class for $\mu$. Let $\{\xi_n\}$ be a sequence of random measures on $S$, i.e., 
a sequence of $\rM$-valued random variables. 
Assume that for every $A \in \mfA$, there exists $c_A \in [0, \infty)$ such that 
$\E[\xi_n(A)] \to c_A$ as $n \to \infty$. 
Then, there exists a unique measure $\mu\in \rM$ such that the mean measure $\E[\xi_n]$ converges vaguely to $\mu$ and $\mu(A) = c_A$ for $ A \in \mfA \cap \bdd_{\mu}$.

Assume further that for every $A\in\mfA$,
\[
	\xi_n(A) \to c_A  \text{ almost surely as } n \to \infty.  
\]
Then $\{\xi_n\}$ converges vaguely to $\mu$ almost surely.
\end{prop}
\begin{proof}
Note that we implicitly assume that $\E[\xi_n(A)] < \infty$ for all $A \in \cA$ and all $n$. Then it follows from the finite covering property that the mean measure $\E[ \xi_n]$ exists for all $n$. By Lemma~\ref{lem:existence-vague-limit}, there exists a unique measure
$\mu$ such that $\E[ \xi_n]$ converges vaguely to $\mu$ as $n \to \infty$, and hence $\mu(A) = c_A$ for $ A \in \mfA \cap \bdd_{\mu}$.  

Now let $\mfA_\mu \subset \mfA$ be a countable convergence-determining class for $\mu$. Then, almost surely
\[
	\xi_{n}(A) \to \mu(A) \text{ as } n \to \infty, \text{ for all } A \in \mfA_\mu,
\]
which implies that the sequence $\{\xi_n\}$ converges vaguely to $\mu$ almost surely. The proof is complete.
\end{proof}

\section{Convergence of persistence diagrams}

\subsection{Proof of Theorem~{\rm\ref{thm:LLN-Betti}}}\label{sec:LLN-Betti}
Let $\Phi$ be a stationary point process  on $\bR^{N}$ having all finite
moments. 
Let $F_q(\Phi, r)$ be the number of $q$-simplices in $K(\Phi,r)$ 
and $F_q(\Phi, r; A)$ be the number of $q$-simplices in $K(\Phi,r)$ 
with at least one vertex in $A \subset \bR^{N}$.
Recall that every $q$-simplex in $K(\Phi, r)$ containing $x$ must lie in
the closed ball $\bar B_{\rho(r)}(x)$. Therefore similar to
\cite[Lemma~3.1]{ysa}, there exists a constant $C_{q,r}$ such that  
\[
\E[F_q(\Phi_{A},r)] \le \E[F_q(\Phi, r; A)] \le C_{q,r} |A| 
\]
for all bounded Borel sets $A$, where $|A|$ is the Lebesgue measure of $A$.

We divide $\La_{mM}$ into $m^{N}$ rectangles that are congruent to 
 $\La_{M}$ and write as follows
\[
 \La_{mM} = \bigsqcup_{i=1}^{m^{N}} (\La_{M} + c_i), 
\]
where $c_i$ is the center of the $i$th rectangle. 
We compare $\K(\Phi_{\La_{mM}})$ with a smaller filtration 
$\K^{\circ}(\Phi_{\La_{mM}}) := \bigsqcup_{i=1}^{m^{N}}\K(\Phi_{\La_{M}+c_i})$. 

Let $\psi(L) = \E[\beta_q^{r,s}(\K(\Phi_{\La_L}))]$ for $r\leq s$.
By \lref{lem:multi_add}, we have 
\begin{align}\label{eq:betti_estimate}
|\beta_q^{r,s}(\K(\Phi_{\La_{mM}})) 
- \beta_q^{r,s}(\K^{\circ}(\Phi_{\La_{mM}}))| 
&\leq 
\sum_{j=q}^{q+1} 
\sum_{i=1}^{m^{N}}
F_j(\Phi_{(\partial
 \La_M)^{(\rho(s))}+c_i}, s).  
\end{align}
Here, for $A \subset \bR^{N}$,
we write $A^{(r)} = \{x \in \bR^{N} : \inf_{y \in A} \|x-y\| \le r\}$.
Since $\E[F_j(\Phi_{(\partial
 \La_M)^{(\rho(s))}+c_i}, s)] = 
O(|(\partial
 \La_M)^{(\rho(s))}+c_i|) = O(M^{N-1})$ as $M \to \infty$, 
we have 
\begin{equation}
\frac{\psi(mM)}{(mM)^{N}} = \frac{\psi(M)}{M^{N}} + O(M^{-1}). 
\label{esti1}
\end{equation}
Moreover, for $L > L'$, 
\begin{align*}
|\beta_q^{r,s}(\K(\Phi_{\La_{L}})) -
\beta_q^{r,s}(\K(\Phi_{\La_{L'}}))| 
&\leq \sum_{j=q}^{q+1} 
F_j(\Phi_{\La_L}, s; \La_L \setminus \La_{L'}) 
\end{align*}
and 
\[
	\E[F_j(\Phi_{\La_L}, s; \La_L \setminus \La_{L'})] = O(|\La_L
 \setminus \La_{L'}|) = 
 O((L-L')L^{N-1}).
\] 
Then, for fixed $M>0$, taking $m \in \N$ such that 
$mM \le L < (m+1)M$, we see that 
\begin{equation}
\frac{\psi(L)}{L^{N}} = \frac{\psi(mM)}{(mM)^{N}}  + O(ML^{-1}). 
\label{esti2}
\end{equation}
It follows from \eqref{esti1} and \eqref{esti2} 
that $\{L^{-N} \psi(L)\}_{L \ge 1}$ is a Cauchy sequence by 
taking sufficient large $M$ first and then $L$, which completes the first part of the proof.

Let us assume now that $\Phi$ is ergodic. Since the arguments are similar to those in the proof of Theorem~3.5 in \cite{ysa}, we only sketch main ideas.  
By the multi-dimensional ergodic theorem, we see that almost surely as $m \to \infty$,
\[
\frac{1}{m^{N}} \beta_{q}^{r,s}(\K^{\circ}(\Phi_{\La_{mM}})) 
= \frac{1}{m^{N}}\sum_{i=1}^{m^{N}}
 \beta_{q}^{r,s}(\K(\Phi_{\La_{M}+c_i}))
\to \E[\beta_{q}^{r,s}(\K(\Phi_{\La_{M}}))] ,
\]
and for $j = q, q+1$,
\[
\frac{1}{m^{N}} \sum_{i=1}^{m^{N}} 
F_j(\Phi_{(\partial \La_M)^{(\rho(s))}+c_i}, s)
\to  \E[F_j(\Phi_{(\partial\La_{M})^{(\rho(s))}},s)] = O(M^{N-1}). 
\]
Remark here that the above equations hold for all except a countable set of $M$ (cf.~\cite[Theorem~1]{ps}).
Therefore, it follows from \eqref{eq:betti_estimate} that
\[
\limsup_{m \to \infty} \frac{\pm 1}{(mM)^{N}}\beta_{q}^{r,s}(\K(\Phi_{\La_{mM}})) 
\le \frac{\pm 1}{M^{N}} 
\E[\beta_{q}^{r,s}(\K(\Phi_{\La_{M}}))] + O(M^{-1}).
\]
The rest of the proof is similar to the last step in the first part by noting that the following laws of large numbers for $F_j(\Phi_{\La_L}, s), j = q, q+1,$ hold (cf.~\cite[Lemma~3.2]{ysa}), 
\[
	 \frac{F_j(\Phi_{\La_L}, s)}{L^j} \to \hat F_j(s) \text{ almost surely as $L \to \infty$.}
\]
This completes the second part of the proof.\qed

\begin{cor}\label{cor:conv}
 Let $\Phi$ be a stationary point process on $\bR^{N}$ having all finite moments, and $\xi_{q,L}$ be the point process on $\Delta$ 
 corresponding to the $q$th persistence diagram for
 $\K(\Phi_{\La_L})$. Then, for every rectangle of the form $R = (r_1,r_2] \times (s_1,s_2], [0, r_1] \times (s_1, s_2] \subset \Delta$, there exists a constant $C_R\in [0, \infty)$ such that
 \[
\frac{1}{L^{N}} \E[\xi_{q,L}(R)] \to C_R \text{ as }L \to \infty.
\]  
In addition, if $\Phi$ is ergodic, then
\[
\frac{1}{L^{N}} \xi_{q,L}(R) \to C_R \text{ almost surely as }L \to \infty.
\]  
\end{cor}
\begin{proof}
It is a direct consequence of Theorem~\ref{thm:LLN-Betti} because for $R = (r_1,r_2] \times (s_1,s_2]$,
\begin{align*}
	&\xi_{q,L}(R) \\&= \beta_q^{r_2,s_1}(\K(\Phi_{\Lambda_L}))
	-\beta_q^{r_2,s_2}(\K(\Phi_{\Lambda_L}))
	+\beta_q^{r_1,s_2}(\K(\Phi_{\Lambda_L}))
	-\beta_q^{r_1,s_1}(\K(\Phi_{\Lambda_L})),
\end{align*}
and for $R = [0, r_1] \times (s_1, s_2]$,
\[
	\xi_{q,L}(R)=\beta_q^{r_1,s_1}(\K(\Phi_{\Lambda_L}))
	-\beta_q^{r_1,s_2}(\K(\Phi_{\Lambda_L})). \qedhere
\]
\end{proof}

\subsection{Proof of Theorem~{\rm\ref{thm:LLN-PD}}}\label{sec:LLN-PD}
Let $S = \Delta = \{(x,y) \in \Rbar^2 : 0 \leq x < y \leq \infty\}$.
Set 
\[
\mfA = \{(r_1, r_2] \times (s_1, s_2], [0,r_1] \times (s_1, s_2] \subset \Delta : \
0 \leq r_1 < r_2 \le s_1 < s_2 \le \infty\}.
\]
We will show in \cref{cor:condition-for-A} that $\mfA$ is a convergence-determining class which satisfies the condition in \pref{prop:LLN}. Theorem \ref{thm:LLN-PD} then follows from \pref{prop:LLN} and \cref{cor:conv}. \qed

\begin{defn}
We call the limiting Radon measure $\nu_{q} \in \rM(\Delta)$ in Theorem~\ref{thm:LLN-PD} the $q$th persistence diagram for a stationary ergodic 
point process $\Phi$ on $\bR^{N}$. 
\end{defn}

 \begin{ex}\label{ex:shifted_lattice}
Let $\Phi$ be a randomly shifted $\bZ^{N}$-lattice with intensity $1$,
i.e., $\Phi = \bZ^{N} + U$, where $U$ is a uniform random variable on the
unit cube $[0,1]^{N}$. Then, $\Phi$ is a stationary ergodic point process in $\bR^{N}$. We compute the limiting persistence diagram $\nu_q$ of the
  {\v C}ech filtration $\C(\Phi)=\{C(\Phi,r)\}_{r\geq 0}$ for
  $q=1,2,\dots, N-1$. 

For this purpose, we introduce a filtration $\bar\C(L)=\{\bar C(L,r)\}_{r\geq0}$ of cubical complexes by
\[
	\bar C(L,r)=\left\{\begin{array}{ll}
		\text{CL}(L,q),&\frac{\sqrt{q}}{2}\leq r < \frac{\sqrt{q+1}}{2},\\
		\text{CL}(L,N),&r\geq\frac{\sqrt{N}}{2},
	\end{array}\right.
\] 
where $\text{CL}(L,N)$ is the cubical complex consisting of all the elementary cubes
in $[0,L]\times\dots\times[0,L]\subset \bR^{N}$, and $\text{CL}(L,q)$ is the $q$-dimensional skeleton of $\text{CL}(L,N)$.
Here a cube $Q=I_1\times\dots\times I_{N}\subset \bR^{N}$ consisting
  of $I_k=[a,a]$ or $I_k=[a,a+1]$ for some $a\in \bZ$ is called an
  elementary cube \cite{kmm}. 
From the stationarity of $\Phi$ and the homotopy equivalence between $\bar C(L,r)$ and $C(\bZ^{N}\cap{[0,L]^{N}},r)$, it suffices to compute the persistence diagram by using  the filtration $\bar\C(L)$.
We also note that $(\sqrt{q}/2,\sqrt{q+1}/2)$ is the only birth-death pair for the $q$th persistence diagram. Therefore, all we need to verify is the multiplicity of that pair with respect to $L$. 

The Euler-Poincar\'{e} formula for $X=\text{CL}(L,q)$ is given by
\begin{align}\label{eq:euler}
	\sum_{k=0}^q(-1)^k|X_k|=\sum_{k=0}^q(-1)^k\beta_k(X).
\end{align}
The number $|X_k|$ of $k$-cells in $X$ is given by (see, e.g., \cite{hs_tutte})
\[
	|X_k|=\sum_{p=k}^{N}\binom{p}{k}S_p(L),
\]
where $S_p(x_1,\dots,x_{N})$ is the elementary symmetric polynomial of degree $p$ and $S_p(L)$ is an abbreviation for $S_p(L,\dots,L)$. 
On the other hand, since $X$ is homotopy equivalent to a wedge sum of $q$-spheres, we have $\beta_0=1$ and $\beta_k=0$ for $k=1,\dots,q-1$. Then, it follows from (\ref{eq:euler}) that 
\[
\beta_q(X)=\sum_{k=0}^q(-1)^{k+q}\sum_{p=k}^{N}\binom{p}{k}S_p(L)+(-1)^{q+1},
\]
and hence
\[
\frac{\beta_q(X)}{L^{N}}=\sum_{k=0}^q(-1)^{k+q}\binom{N}{k}+O(L^{-1})
  =\binom{N-1}{q}+O(L^{-1}).
\]
Therefore, the limiting persistence diagram is given by 
\[
\nu_{q} = \binom{N-1}{q}\delta_{(\sqrt{q}/2,\sqrt{q+1}/2)}.
\]
 \end{ex}

\subsection{The support of $\nu_q$}
\label{sec:positivity}

In this section, we give some sufficient conditions both on $\kappa$ and
$\Phi$ to ensure the positivity of the limiting measure $\nu_q$.
We use the following stability result on persistence diagrams of
$\kappa$-filtrations (cf.~\cite{cdo, ceh}). 
Here, the persistence diagram of the $\kappa$-filtration on
$\Xi\in\finite(\bR^{N})$ is simply denoted by $D_q(\Xi)$.  
\begin{lem}\label{lem:continuity}
 Assume that $\kappa$ is Lipschitz continuous with respect to the Hausdorff distance, i.e., there exists a constant $c_\kappa$ such that
\[
|\kappa(\sigma)-\kappa(\sigma')|\leq c_\kappa d_H(\sigma,\sigma')
\] 
for $\sigma,\sigma'\in\finite(\bR^{N})$. 
Then, for $\Xi,\Xi'\in\finite(\bR^{N})$,
\[
	d_B(D_q(\Xi), D_q(\Xi')) \le c_\kappa d_H(\Xi, \Xi'). 
\]
Here $d_B$ and $d_H$ denote the bottleneck distance and the Hausdorff
 distance, respectively.  
\end{lem}
See Appendix C for the detail, where we recall definitions of $d_B$ and
$d_H$ and give a proof of a generalization
of this lemma.  

Next we introduce the notion of \textit{marker} which is
a finite point configuration for finding a specified point in $\Delta$. 

 \begin{defn}
  Let $\La$ be a bounded Borel set in $\bR^{N}$ and $(b,d) \in \Delta$. 
  We say that $\fp \in \finite(\bR^{N})$ is a \textit{$(b,d)$-marker in
  $\La$}  
 for the $q$th persistent homology ($\phq$) if (i) $\fp \subset \La$ and
 (ii) for any $\Phi \in \finite(\bR^{N})$ 
 \begin{equation}
 \xi_q(\Phi_{\La^c} \sqcup \fp)(\{(b,d)\}) \ge 
 \xi_q(\Phi_{\La^c})(\{(b,d)\}) + 1.  
\label{eq:marker}  
 \end{equation}
 Here $\La^c$ denotes the complement of $\La$ in $\R^N$.
  For a subset $A \subset \Delta$, we also say that $\fp$ is an $A$-marker
  in $\La$
 if there exists $(b,d) \in A$ such
 that $\fp$ is a $(b,d)$-marker in $\La$. 
 \end{defn}

 \begin{ex} \label{ex:marker}
 \begin{itemize}
 \item[(i)]	Assume that a point	 $(b,d) \in \Delta$ is realizable by $\fp \in
  \finite(\bR^{N})$. 
  Then, there exists $M_0 >0$ such that 
  $\fp$ is a $(b,d)$-marker in $\La_M$ for any $M \ge M_0$
  because $\fp$ is enough isolated from $\La_M^c$ for sufficiently large
  $M$. 
  \item[(ii)] We note that $(1/2, \sqrt{2}/2) \in \Delta$
   is realized in $\mathrm{PH}_1$
   by 
   \[
   	\{(0,0), (1,0), (0,1), (1,1)\} \in \finite(\bR^2)
\] in
   the \cech or Rips filtration.  
   It is easy to see that for each $c \in \La_1$ 
   $\fp_M = \sum_{x \in \Z^2 \cap \La_M} \delta_{x+c}$ is a $(1/2,
   \sqrt{2}/2)$-marker in $\La_{M}$ for any sufficiently large $M$, for example, $M = 3$.
\end{itemize}
\end{ex} 

\begin{thm}\label{thm:support}
 Let 
 \[
A_{q,\eps,(b,d)}
 := \bigcup_{M=1}^{\infty} \{
 \text{$\Phi_{\La_{M}}$ is a $B_{\eps}((b,d))$-marker in $\La_{M}$
 for $\phq$}\} 
 \]
 and
\[
S_{q, \eps} := \{(b,d) \in \Delta : 
\PP(A_{q,\eps,(b,d)})>0 \}, \quad 
 S_q := \bigcap_{\eps > 0} S_{q, \eps}. 
\]
 Then, $S_q \subset \support \nu_q$. 
\end{thm}

Before proving Theorem~\ref{thm:support}, we give a lower bound for
$\nu_q$. 
\begin{lem} For a closed set $A \subset \Delta$, 
\begin{equation}
 \nu_q(A)
  \ge \frac{1}{M^{N}} \PP(\text{$\Phi_{\La_M}$ is an $A$-marker in $\La_M$
  for $\phq$}). 
\label{eq:nuqa} 
\end{equation}
\end{lem}
  \begin{proof}
   Let $\La$ be a bounded Borel set in $\bR^{N}$ and $(b,d) \in \Delta$.
   If one could find disjoint subsets
  $\La^{(1)}, \dots, \La^{(k)} \subset \La$ such that
  $\Phi_{\La^{(i)}}$ is a $(b,d)$-marker in $\La^{(i)}$ for each $i$,
   then 
 \begin{equation}
  \xi_q(\Phi_{\La})(\{b,d\}) \ge k.
   \label{eq:lowerbound0}
 \end{equation}
Indeed, by using \eqref{eq:marker} successively, we have
 \[
   \xi_q(\Phi_{\La})(\{(b,d)\}) \ge 
   \xi_q(\Phi_{\La \setminus \cup_{j=1}^k \La^{(j)}})(\{(b,d)\})
  +k \ge k.
 \]
For $L > M >0$ and $m = \lfloor L/M \rfloor$, 
  we claim that 
 \begin{equation}
  \xi_q(\Phi_{\La_L})(A) \ge \sum_{i=1}^{m^{N}} \one_{\{\text{$\Phi_{\La_M + c_i}$ is an $A$-marker in $\La_M + c_i$ for $\phq$}\}},  
\label{eq:lowerbound}  
 \end{equation}
  where $c_i \in \La_L, i=1,2,\dots,m^{N}$ are chosen so that
  $\La_L \supset \sqcup_{i=1}^{m^{N}} (\La_M + c_i)$.
If the right-hand side of \eqref{eq:lowerbound}  
   is equal to $k$, we have disjoint subsets
   $I_j \subset \{c_1,c_2,\dots, c_{m^{N}}\}, j=1,2,\dots, J$ with
   $\sum_{j=1}^J |I_j| = k$ such that
   for every $j=1,2,\dots,J$, $\Phi_{\La_M + c}$ is a $(b_j,d_j)$-marker
  in $\La_M + c$ for $\phq$ 
   for any $c \in I_j$.
   Here $(b_j,d_j) \in A, j=1,2,\dots, J$ are all distinct. 
   From \eqref{eq:lowerbound0}, we have 
  \[
  \xi_q(\Phi_{\La_L})(A) \ge \sum_{j=1}^J \xi_q(\Phi_{\La_L})(\{(b_j, d_j)\}) 
  \ge \sum_{j=1}^J |I_j| = k, 
  \]
   which implies \eqref{eq:lowerbound}. 

   For a closed set $A \subset \Delta$, from \eqref{eq:lowerbound}, we
   obtain 
\begin{align*}
 \nu_q(A) &\ge \limsup_{L \to \infty}
 \frac{1}{L^{N}} \E[\xi_q(\Phi_{\La_L})](A)
 \\
 &\ge  \limsup_{L \to \infty}
 \frac{1}{L^{N}} \sum_{i=1}^{m^{N}} \PP(\text{$\Phi_{\La_M + c_i}$ is an
 $A$-marker in $\La_M + c_i$ for $\phq$})\\
 &= 
 \frac{1}{M^{N}} \PP(\text{$\Phi_{\La_M}$ is an $A$-marker in $\La_M$ for
 $\phq$}). 
\end{align*}
  This completes the proof.
  \end{proof}
 
\begin{proof}[Proof of Theorem~\ref{thm:support}]
 If $(b,d) \in S_q$, then
 for every $\eps > 0$, there exists $M= M_{\eps} \in \N$ such that 
\[
\PP(\text{$\Phi_{\La_M}$ is a $B_{\eps}((b,d))$-marker in
 $\La_M$ for $\phq$})>0. 
\]
From \eqref{eq:nuqa}, we see that $\nu_q({\bar B_{\eps}((b,d))})
 > 0$ for any $\eps>0$, which implies $(b,d) \in \support \nu_q$. 
 Therefore, $S_q \subset \support \nu_q$. 
\end{proof}

For a bounded set $\La \subset \bR^{N}$, the restriction $\rN(\La)$ of
$\rN(\bR^{N})$ on $\La$ can be 
identified with $\bigcup_{k=0}^{\infty} \La^k/ \sim$, where $\sim$ is the equivalence relation induced by permutations on coordinates. 
Let $\Pi$ be the probability distribution of homogeneous Poisson point process
with unit intensity. It is clear that the local densities, which are sometimes
called Janossy densities, of the restriction of $\Pi$ on $\La$ is given by 
\[
\Pi|_{\La} =
\begin{cases}
 \frac{e^{-|\La|}}{k!} dx_1 dx_2 \cdots dx_k & \text{ on $\La^k$}, \\
 e^{-|\La|} & \text{on $\La^0 = \{\emptyset\}$}.  
\end{cases}
\]
For a probability measure $\Theta$ on $\rN(\bR^{N})$, 
if $\Theta|_{\La}$ is absolutely continuous with respect to $\Pi|_{\La}$
for a bounded set $\La$, then $\Theta|_{\La}$ is absolutely continuous
with respect to the Lebesgue measure on each $\La^k$ for every $k$;
thus the Radon-Nikodym density $d\Theta|_{\La} / d\Pi|_{\La}$ is
defined a.e. on $\La^k$ for every $k$.  

\begin{proof}[Proof of Theorem~\ref{thm:positivity0}]
 Assume that $(b,d) \in R_q$ and it is realizable by
 $\{y_1,\dots, y_m\}$. From continuity of persistence diagram in
 Lemma~\ref{lem:continuity}, 
 for any $\eps > 0$ there exists $\delta>0$ such that 
 $\xi_q(\{z_1,\dots,z_m\})(B_{\eps}(\{(b,d)\})) \ge 1$ 
 for any $(z_1,\dots,z_m) \in B_{\delta}(y_1) \times \cdots \times
 B_{\delta}(y_m)$.
 From Example~\ref{ex:marker}(i), there exists $M \in \N$ such that
 any $\{z_1,\dots,z_m\}$ is a $B_{\eps}((b,d))$-marker in
 $\La_M$. Hence, we see that 
 \begin{align*}
\lefteqn{ \PP(\text{$\Phi_{\La_M}$ is a $B_{\eps}((b,d))$-marker in
 $\La_M$ for $\phq$})} \\
 &\ge \Theta|_{\La_M}(\cap_{i=1}^m \{\Phi(B_{\delta}(y_i) = 1)\} \cap
  \{\Phi(\La_M \setminus \cup_{i=1}^m B_{\delta}(y_i))=0\} \\
  &= \frac{e^{-|\La_M|}}{m!} \int_{B_{\delta}(y_1) \times \cdots \times
  B_{\delta}(y_m)} f_{\La_M}(z_1,\dots,z_m) dz_1 \cdots dz_m \\
  &> 0, 
 \end{align*}
 where $f_{\La_M} = d\Theta|_{\La_M} / d\Pi|_{\La_M}$. 
 Hence, $R_q \subset S_q\subset \support \nu_q$ by Theorem~\ref{thm:support}.
 Since $\support \nu_q \subset \overline{R_q}$ as mentioned after
 Example~\ref{rem:realizable},  
 we conclude that $\support \nu_q = \overline{R_q}$. 
\end{proof}

\section{Central limit theorem for persistent Betti numbers}\label{sec:CLT-pbetti}
In this section, let $\Phi = \cP$ be a homogeneous Poisson point process with unit intensity, and we prove Theorem~\ref{thm:CLT}.
The idea is to apply a result in \cite{py} which shows a central limit theorem for a certain class of functionals defined on Poisson point processes.

We here summarize necessary properties for functionals to achieve the central limit theorem. 
First of all, let us consider a sequence $\{W_n\}$ of Borel subsets in $\bR^{N}$ satisfying the following conditions:
\begin{itemize}
\item[(A1)] $|W_n|=n$ for all $n\in \N$;
\item[(A2)] $\bigcup_{n\geq 1}\bigcap_{m\geq n} W_m=\bR^{N}$;
\item[(A3)] $\lim_{n\rightarrow\infty}|(\partial W_n)^{(r)}|/n=0$ for all $r>0$;
\item[(A4)] there exists a constant $\gamma > 0$ such that $\diam(W_n)\leq \gamma n^\gamma$.
 \end{itemize}
Given such a sequence, let $\cW = \cW(\{W_n\})$ be the collection of all
subsets $A$ in $\bR^{N}$ of the form
$A = W_n + x$ for some $W_n$ in the sequence and some point $x \in \bR^{N}$.

Let $H$ be a real-valued functional defined on $\finite(\bR^{N})$.  
The functional $H$ is said to be {\it translation invariant} if it satisfies $H( \cX + y) = H(\cX)$ for any $\cX\in\finite(\bR^{N})$ and $y \in \bR^{N}$. 
Let $D_0$ be the add one cost function
\[
	D_0 H(\cX) = H(\cX \cup \{0\}) - H(\cX),\quad\cX\in\finite(\bR^{N}),
\]
which is the increment in $H$ caused by inserting a point at the origin.
The functional $H$ is {\it weakly stabilizing} on $\cW$ if there exists
a random variable $D(\infty)$
such that $D_0H(\cP_{A_n})\overset{a.s.}{\longrightarrow} D(\infty)$ as
$n\rightarrow \infty$ for any sequence $\{A_n\in\cW\}_{n\geq 1}$ tending to $\bR^{N}$.
The {\it Poisson bounded moment condition} on $\cW$ is given by
	\[
		\sup_{0\in A \in \cW} \Ex[ (D_0 H(\cP_A))^4] < \infty.
	\]
Then, we restate Theorem~3.1 in \cite{py} in the following form.
\begin{lem}[{\cite[Theorem~3.1]{py}}]\label{lem:py}
 Let $H$ be a real-valued functional defined on $\finite(\bR^{N})$.
 Assume that $H$ is translation invariant and weakly stabilizing on $\cW$, and satisfies the Poisson bounded moment condition.  
Then, there exists a constant $\sigma^2 \in [0, \infty)$ such that $n^{-1} \Var[H(\cP_{W_n})] \to \sigma^2$ and 
\[
	\frac{H(\cP_{W_n}) - \Ex[H(\cP_{W_n})]}{n^{1/2}} \dto \cN(0, \sigma^2) \text{ as } n \to \infty.
\]
\end{lem}

By using Lemma \ref{lem:py}, we prove the following theorem.
\begin{thm}\label{thm:CLT_general}
Let $\Phi = \cP$ be a homogeneous Poisson point process with unit intensity. Assume that the sequence $\{W_n\}$ satisfies \rm{(A1)--(A4)}.
Then for any $0 \le r \le s < \infty$,
	\[
		\frac{\beta_q^{r,s} (\K(\cP_{W_n})) - \E[\beta_q^{r,s} (\K(\cP_{W_n})) ]}{n^{1/2}} \dto \cN(0, \sigma_{r,s}^2)\text{ as } n \to \infty.
	\]
\end{thm}

In particular, Theorem~\ref{thm:CLT} is derived from this theorem by taking $W_n=\Lambda_{L_n}$ with $L_n=n^{1/N}$.

For the proof of Theorem~\ref{thm:CLT_general}, the essential part is to
show the weak stabilization of the persistent Betti number
$\beta^{r,s}_q(\K(\cdot))$ as a functional on $\finite(\bR^{N})$, on which
we focus below. 

We remark that, for almost surely, the Poisson point process $\cP$ consists of infinite points in $\bR^{N}$ which do not have accumulation points. In view of this property, we first show a stabilization of persistent Betti numbers in the following deterministic setting.

\begin{lem} \label{lem:stabilization}
Let $P$ be a set of points in $\bR^{N}$ without accumulation points. 
Then,  for each fixed $r\leq s$, there exist constants $D_\infty$ and $R>0$ such that
\begin{align*}
D_0\beta_q^{r,s}(\K(P_{\bar B_a(0)}))=D_\infty
\end{align*}
for all $a\geq R$. 
\end{lem}

\begin{proof} 
Let $P'=P\cup\{0\}$. 
Let $K_{r,a}=K(P_{\bar B_a(0)},r)$ be the simplicial complex defined on
 $P_{\bar B_a(0)}$ with parameter $r$, and similarly let $K_{r,a}'=K(P'_{\bar B_a(0)},r)$.  

From the definition (\ref{eq:rsbetti}), $D_0 \beta^{r,s}_q(\K(P_{\bar B_a(0)}))$
can be expressed as
\begin{align*}
&D_0 \beta^{r,s}_q(\K(P_{\bar B_a(0)})) \\
&=\rank\frac{Z_q(K'_{r,a})}{Z_q(K'_{r,a})\cap B_q(K'_{s,a})}-
\rank\frac{Z_q(K_{r,a})}{Z_q(K_{r,a})\cap B_q(K_{s,a})}\\
&=(\rank Z_q(K'_{r,a})-\rank Z_q(K_{r,a}))\\
&\quad-(\rank Z_q(K'_{r,a})\cap B_q(K'_{s,a})-\rank Z_q(K_{r,a})\cap B_q(K_{s,a})).
\end{align*}
Hence, it suffices to show the stabilization with respect to $a$ for $\rank Z_q(K_{r,a})$ and $\rank (Z_q(K_{r,a})\cap B_q(K_{s,a}))$ separately.

Let us study $\rank Z_q(K_{r,a})$. Since the rank takes non-negative
 integer values, we show the bounded and the non-decreasing properties.
 First of all, note that $K_{r,a}\subset K'_{r,a}$,
 and hence $Z_q(K_{r,a})\subset Z_{q}(K'_{r,a})$. Let us express $K'_{r,a}$ as a disjoint union $K'_{r,a}=K_{r,a}\sqcup K^0_{r,a}$, where $K^0_{r,a}$ is the set of simplices having the point 0, and let $K^0_{r,a,q}=\{\sigma\in (K'_{r,a})_q : 0\in\sigma\}$. 

Let $\partial_{q,a}$ and $\partial'_{q,a}$ be the $q$th boundary maps on $K_{r,a}$ and $K'_{r,a}$, respectively. Then, we can obtain the following block matrix form
\begin{align}\label{eq:dashboundary}
\partial'_{q,a}=
\left[\begin{array}{cc}
M_{1,\rho} & \mathbf{0}\\
M_{2,\rho} & \partial_{q,a}\\
\end{array}\right],
\end{align}
where the first columns and rows are arranged by the simplices in $K^0_{r,a,q}$ and $K^0_{r,a,q-1}$, and the second columns and rows correspond to the simplices in $K_{r,a}$.

Recall that any simplex $\sigma \in K(P, r)$ containing the point 0 is included in $\bar B_{\rho(r)}(0)$.
Hence, the set $K^0_{r,a,q}$ becomes independent of $a$ for $a\geq\rho(r)$, which we denote by $K^0_{r,*,q}$. 
From this observation and Lemma \ref{lem:col_add} applied to the matrix form (\ref{eq:dashboundary}),  
we have
\begin{align*}
\rank Z_q(K'_{r,a})-\rank Z_q(K_{r,a})\leq |K^0_{r,a,q}|= |K^0_{r,*,q}|, 
\end{align*}
 which gives the boundedness. 

In order to show the non-decreasing property, let us consider a homomorphism defined by
\begin{align*}
f \colon \frac{Z_q(K'_{r,a_1})}{Z_q(K_{r,a_1})}\ni [c]\longmapsto [c]\in \frac{Z_q(K'_{r,a_2})}{Z_q(K_{r,a_2})}
\end{align*}
for $a_1\leq a_2$.
This map is well-defined because $Z_q(K_{r,a_1})\subset Z_q(K_{r,a_2})$
 and $Z_q(K'_{r,a_1})\subset Z_q(K'_{r,a_2})$ hold. Suppose that $f[c]=0$. Then, the cycle $c\in Z_q(K'_{r,a_1})$ is in $Z_q(K_{r,a_2})$. It means that the $q$-simplices consisting of $c$ do not contain the point $0$, and hence $c\in Z_q(K_{r,a_1})$. This shows that the map $f$ is injective. From this observation, we have the inequality
\begin{align*}
\rank Z_q(K'_{r,a_1})/Z_q(K_{r,a_1}) \leq \rank Z_q(K'_{r,a_2})/Z_q(K_{r,a_2}),
\end{align*}
which leads to the desired non-decreasing property. This finishes the proof of the stabilization of $\rank Z_q(K_{r,a})$.

Let us study the stabilization of $\rank (Z_q(K_{r,a})\cap
 B_q(K_{s,a}))$. The strategy is basically the same as above. It follows
 from Lemma \ref{lem:rank_lemma} that 
\begin{align*}
\rank \frac{Z_q(K'_{r,a})\cap B_q(K'_{s,a})}{Z_q(K_{r,a})\cap B_q(K_{s,a})}\leq \rank\frac{Z_q(K'_{r,a})}{Z_q(K_{r,a})} + \rank\frac{B_q(K'_{s,a})}{B_q(K_{s,a})}.
\end{align*}
Then, from the same reasoning used in  $\rank Z_q(K_{r,a})$, we have the stabilization $|K^0_{s,a,q+1}|=|K^0_{s,*,q+1}|$ for large $a$. Hence, we have the boundedness
\begin{align*}
&\rank Z_q(K'_{r,a})\cap B_q(K'_{s,a})-\rank Z_q(K_{r,a})\cap B_q(K_{s,a})\leq |K^0_{r,*,q}| +|K^0_{s,*,q+1}|.
\end{align*}
Similarly, we can show the injectivity of the map
\begin{align*}
f \colon \frac{Z_q(K'_{r,a_1})\cap B_q(K'_{s,a_1})}{Z_q(K_{r,a_1})\cap B_q(K_{s,a_1})}\longrightarrow \frac{Z_q(K'_{r,a_2})\cap B_q(K'_{s,a_2})}{Z_q(K_{r,a_2})\cap B_q(K_{s,a_2})}, 
\quad f[c]=[c]
\end{align*}
from which the non-decreasing property follows. This completes the proof of the lemma.  
\end{proof}

\begin{prop}\label{prop:weakstab}
The functional $\beta^{r,s}_q(\K(\cdot))$ is weakly stabilizing.
\end{prop}
\begin{proof}
Let $R>0$ be chosen as in Lemma \ref{lem:stabilization} and let $\{A_n\in\cW\}_{n\geq 1}$ be a sequence tending to $\bR^{N}$. Then, there exists $n_0\in \N$ such that $B_R(0)\subset A_n$ for all $n\geq n_0$. 

For $n\geq n_0$, 
let us set $L_{r,n}=K(\cP_{A_n},r)$.
Then, since $A_n$ is bounded, there exists $a>R$ such that 
\begin{align*}
B_R(0)\subset A_n\subset B_a(0).
 \end{align*}
 Then, as in the same way used for showing the injectivity in the proof of Lemma \ref{lem:stabilization}, we can show 
 \begin{align*}
 \frac{Z_q(K'_{r,R})}{Z_q(K_{r,R})}\subset
 \frac{Z_q(L'_{r,n})}{Z_q(L_{r,n})}\subset
 \frac{Z_q(K'_{r,a})}{Z_q(K_{r,a})},
 \end{align*}
 where $K_{r,a}=K(\cP_{\bar B_a(0)},r)$ as before. 
 Since the ranks of ${Z_q(K'_{r,R})}/{Z_q(K_{r,R})}$ and
  ${Z_q(K'_{r,a})}/{Z_q(K_{r,a})}$ are equal, for all $n \ge n_0$ 
 \begin{align*}
 \rank Z_q(K'_{r,R})-\rank Z_q(K_{r,R})=
 \rank Z_q(L'_{r,n})-\rank Z_q(L_{r,n}).
 \end{align*}
 We  can also show that $\rank Z_q(L'_{r,n})\cap B_q(L'_{s,n})-\rank Z_q(L_{r,n})\cap B_q(L_{s,n})$ is invariant for $n\geq n_0$ in a similar manner. 
 This completes the proof.
 \end{proof}

 \begin{proof}[Proof of Theorem~{\rm \ref{thm:CLT_general}}]
For fixed $r\le s$, we regard the persistent Betti number $\beta^{r,s}_q(\K(\cdot))$ as a functional on $\finite(\bR^{N})$, and check the three conditions stated in Lemma \ref{lem:py}.
 First, the translation invariance is obvious, because $\kappa$ is translation invariant. 
 Next, let us consider the Poisson bounded moment condition on $\cW$. We note the following estimate:
 \begin{align*}
 |D_0 \beta_q^{r,s} (\K(\cP_A)) |
 &= |\beta_q^{r,s} (\K(\cP_A\cup\{0\}))-\beta_q^{r,s} (\K(\cP_A))|\\
 &\le \sum_{j=q,q+1}|K_j(\cP_A\cup\{0\},s)\setminus K_j(\cP_A,s)|\\
 &\le \sum_{j=q,q+1}F_j(\cP_{\bar B_{\rho(s)}(0)},s)<\infty.
 \end{align*}
 Here, the second inequality follows from Lemma \ref{lem:multi_add},
  and the boundedness of the last expression is shown by the finite
  moments of the Poisson point process on $\bar B_{\rho(s)}(0)$. 
  We showed the weak stabilization in Proposition~\ref{prop:weakstab}.
 The proof of Theorem \ref{thm:CLT_general} is now complete.  
 \end{proof}

 \section{Conclusions}\label{sec:conclusion}
 In this paper, we studied a convergence of persistence diagrams and
 persistent Betti numbers for stationary point processes, and a central limit
 theorem of persistent Betti numbers for homogeneous Poisson point
 process. Several important problems are still yet to be solved. 

 \begin{enumerate}
  \item We showed the existence of limiting persistence diagram for 
	simplicial complexes built over stationary ergodic point
	processes. Such convergence results can be expected for more general
	random simplicial/cell complexes studied in \cite{hs_frieze,
	hs_tutte}. It would also be important to
	investigate the rate of convergence from the statistical and
	computational point of view. 
  \item Attractiveness/repulsiveness of point processes are reflected
	on persistence diagrams (see Figure~\ref{fig:pppd}).
	For example, the mass of the limiting persistence diagram
	$\nu_q$ for negatively correlated point process
	seems to become more concentrated than that for positively
	correlated point process. 
  \item The moments of the limiting persistence diagram 
	$\int_{\Delta} |y-x|^n \nu_q(dxdy)$, should be studied. 
 	 Other properties of limiting persistence diagrams such
 	 as continuity, absolute continuity/singularity,
 	 comparison etc. should also be 
        investigated thoroughly for practical purposes (cf. \cite{kfh, pnas}). 
 \item The central limit theorem for persistent Betti numbers (even for
       usual Betti numbers) is only proved for Poisson point processes.
       It could be extended to more general stationary point processes. 
       We also expect that a scaled persistence diagram converges to
       a Gaussian field on $\Delta$. 
\end{enumerate}

\appendix
\section{Convergence-determining class for vague convergence}\label{app:convergence}
We provide a sufficient condition for a class of $\bdd$-sets to be a convergence determining class for vague convergence. 
We use the same notations as in Section \ref{sec:rm}. 
Assume that a class $\mfA \subset \bdd$ is closed under finite intersections. Let us define
\[
	\cR(\mfA) = \bigg\{\bigcup_{finite} A_i: A_i \in \mfA \bigg\}.
\] 
Then $\cR(\mfA)$ is closed under both finite intersections and finite unions. Furthermore, 
if $\mu_n(A) \to \mu(A)$ for all $A \in \mfA$, 
then so does for all $A \in \cR(\mfA)$, because 
\[
\mu(\bigcup_{i = 1}^m A_i) = \sum_i \mu(A_i) - \sum_{i\neq j} \mu(A_i \bigcap A_j) +
\cdots + (-1)^{m-1} \mu(\bigcap_{i=1}^m  A_i).   
\] 
\begin{lem}\label{lem:open-closed}
Assume that a class $\mfA$ is closed under finite intersections, and that 
\begin{itemize}
	\item[\rm (i)] each open set $G \in \bdd$ is a countable union of $\cR(\mfA)$-sets, and
	\item[\rm(ii)] each closed set $F \in \bdd$ is a countable intersection of $\cR(\mfA)$-sets.
\end{itemize}
If $\mu_n(A) \to \mu(A)$ for all $A \in \mfA$, then $\mu_n$  converges vaguely to $\mu$. In particular, the class $\mfA$ is a convergence-determining class for $\mu$ provided that $\mfA \subset \mfB_\mu$.
\end{lem}
\begin{proof}
	Let $G \in \bdd$ be an open set. By assumption, there are sets $A_i \in \cR( \mfA)$ such that 
\[
	G = \bigcup_{i = 1}^\infty A_i.
\]
Given $\eps > 0$, choose an $m$ such that 
\[
	\mu(\bigcup_{i = 1}^m A_i) > \mu(G) - \eps.
\]
Then we have
\[
	\mu(G)-\eps < \mu(\bigcup_{i = 1}^m A_i) = \lim_{n \to \infty} \mu_n(\bigcup_{i = 1}^m A_i) \le \liminf_{n \to \infty} \mu_n(G).
\]
Since $\eps$ is arbitrary, we get 
\[
	\mu(G) \le \liminf_{n \to \infty} \mu_n(G).
\]

Now for a closed set $F \in \bdd$, take $A_i \in \cR(\mfA)$ such that 
\[
	F = \bigcap_{i = 1}^\infty A_i.
\]
Since $A_i \in \bdd$, for given $\eps > 0$, we can choose $m$ large enough such that  
\[
	\mu(F) + \eps > \mu(\bigcap_{i = 1}^m A_i).
\]
Then, it follows from $\cap_{i = 1}^m A_i \in \cR(\mfA)$ that
\[
	\mu(F) +\eps > \mu(\bigcap_{i = 1}^m A_i) = \lim_{n \to \infty} \mu_n(\bigcap_{i = 1}^m A_i) \ge \limsup_{n \to \infty} \mu_n(F) .
\]
Letting $\eps \to 0$, we get
\[
	\mu(F)  \ge \limsup_{n \to \infty} \mu_n(F).
\]
Therefore, the conclusion follows from Lemma~\ref{lem:equivalent-condition-vague-convergence}.
\end{proof}

For given $\mfA$, let $\mfA_{x,\eps}$ be the class of $\mfA$-sets
satisfying $x \in A^{\circ} \subset A \subset B_{\eps}(x)$, where $A^{\circ}$ is the interior of $A$. 
Let $\partial \mfA_{x, \eps}$ be the class of their boundaries, i.e., $\partial \mfA_{x, \eps} = \{\partial A : A \in \mfA_{x, \eps} \}$.

The following theorem gives a
sufficient condition for a
class $\mfA$ to be a convergence-determining class for
vague convergence of Radon measures (see Theorem 2.4 in
\cite{billingsley} for an analogous result on weak convergence of probability measures). 
\begin{thm}\label{thm:convergence-determining-class-criterion}
	Suppose that $\mfA$ is closed under finite intersections and, for each $x \in S$ and
 $\eps > 0$, $\partial \mfA_{x, \eps}$ contains either
 $\emptyset$ or uncountably many disjoint sets. Then, $\mfA$ is a
 convergence-determining class. Moreover, for any measure $\mu \in \rM$,
 $\mfA$ contains a countable convergence-determining class for $\mu$.
\end{thm}
\begin{proof}
	Fix an arbitrary $\mu\in \rM$, and let $\mfA_\mu = \mfA \cap \bdd_\mu$ be
 the class of $\mu$-continuity sets in $\mfA$. Since  
	\[
		\partial (A \cap B) \subset (\partial A) \cup (\partial B),
	\]
$\mfA_\mu$ is again closed under finite intersections. 

Let $G \in \bdd$ be an open set. For $x \in G$, choose $\eps > 0$ such that $B_\eps(x) \subset G$. By the assumption,
 if $\partial \mfA_{x, \eps}$ does not contain 
 $\emptyset$, then it must contain uncountably many disjoint sets.
Hence, in either case, $\partial\mfA_{x,\eps}$ contains a set $A_x$ of $\mu$-measure $0$, or $A_x \in \mfA_\mu$.  Therefore, $G$ can be written as  
\[
	G = \bigcup_{x \in G} A_x^\circ = \bigcup_{x \in G} A_x. 
\]
Since $S$ is a separable metric space, there is a countable subcollection $\{A_{x_i}^\circ\}$ of $\{A_x^\circ : x \in G\}$ which covers $G$, namely,
\[
	G = \bigcup_{i=1}^\infty A_{x_i}^\circ.
\]

Let $\{G_i\}_{i = 1}^\infty$ be a countable basis of $S$. For each $i$, we have just shown that there are countable sets $\{A_{i, j}\}_{j = 1}^\infty \subset \mfA_\mu$ such that 
\[
	G_i = \bigcup_{j = 1}^\infty A_{i, j}^\circ = \bigcup_{j = 1}^\infty A_{i, j}.
\]
Set 
\[
	\mfA'_\mu = \{\cap_{finite} A_{i, j}\}.
\]
Then $\mfA'_\mu \subset \mfA_\mu$ is countable and closed under finite intersections. The remaining task is to show that $\mfA'_\mu$ satisfies the two conditions in Lemma \ref{lem:open-closed}. The condition for open sets is clear from the construction of $\mfA'_\mu$. 

Next, let $F \in \bdd$ be a closed (thus compact) set. For each $\eps > 0$, let
\[
	F^{(\eps)} = \{x \in S: d(x, F) = \inf_{y \in F} \rho(x,y) \le \eps\}.
\]
Then $F = \cap_{p = 1}^\infty F^{(\frac 1p)}$. We claim that, for each $\eps>0$, there exist $m = m(\eps)$ and a collection of sets $\{C_{k}\}_{k = 1}^m\subset \mfA'_{\mu}$  such that 
\[
	F \subset \bigcup_{k = 1}^m C_{k}  \subset F^{(\eps)}.
\]
Indeed, for each $x \in F$, there is a pair $(i_x,j_x)$ such that $x \in A_{i_x, j_x}^\circ \subset A_{i_x, j_x} \subset G_{i_x}  \subset B_{\eps}(x)$. Let $C_x = A_{i_x, j_x}$. Then
\[	
	F \subset \bigcup_{x \in F} C_x^{\circ}.
\]
Since $F$ is compact, there is a finite collection  $\{C_{x_k}^{\circ}\}_{k = 1}^m$ such that
\[
	F \subset \bigcup_{k = 1}^m C_{x_k}^{\circ}.
\]
Finally, note that $C_{x_k}^{\circ} \subset C_{x_k} \subset F^{(\eps)}$, we have 
\[
	F \subset \bigcup_{k = 1}^m C_{x_k} \subset F^{(\eps)}.
\]
Therefore, the condition for closed sets in Lemma \ref{lem:open-closed} is satisfied, which completes the proof of Theorem~\ref{thm:convergence-determining-class-criterion}.
\end{proof}

\begin{cor}\label{cor:condition-for-A}
The class 
 \[
	\mfA=\{(r_1, r_2]\times(s_1, s_2], [0, r_2] \times (s_1,  s_2] \subset \Delta : 0\le r_1 \le r_2 \le s_1 \le s_2 \le \infty\}
 \]
satisfies the conditions of Proposition~{\rm\ref{prop:LLN}}, namely, for any measure $\mu$, it  contains a countable convergence determining class for $\mu$.
\end{cor}
\begin{proof}
 It suffices to check the conditions in 
 Theorem~\ref{thm:convergence-determining-class-criterion}.
 It is clear that $\mfA$ is closed under finite intersection and
 $\partial \mfA_{x, \eps}$ contains countably many disjoint sets
 for any $x \in \Delta$ and $\eps >0$. 
\end{proof}

\section{Simplicial complex and Homology}\label{ap:topology}

\subsection{Simplicial complex}\label{sec:sc}
We first introduce a combinatorial object called simplicial complex. Let $P=\{1,\dots,n\}$ be a finite set (not necessary to be points in a metric space). A {\em simplicial complex} with the vertex set $P$ is defined by a collection $K$ of subsets in $P$ satisfying the following properties:
\begin{itemize}
\item[(i)] $\{i\}\in K$ for $i=1,\dots,n$, and
\item[(ii)] if $\sigma\in K$ and $\tau\subset \sigma$, then $\tau\in K$.
\end{itemize}

Each subset $\sigma$ with $q+1$ vertices is called a $q$-simplex. We denote the set of $q$-simplices by $K_q$. A subcollection $T\subset K$ which also becomes a simplicial complex is called a subcomplex of $K$.

\begin{ex}\label{exam:sc}
Figure \ref{fig:sc} shows two polyhedra of simplicial complexes
\begin{align*}
&K=\{
\{1\},
\{2\},
\{3\},
\{1,2\},
\{1,3\},
\{2,3\},
\{1,2,3\}\},\\
&T=\{
\{1\},
\{2\},
\{3\},
\{1,2\},
\{1,3\},
\{2,3\}\}.
\end{align*}

\begin{figure}[htbp]
 \begin{center}
  \includegraphics[width=50mm]{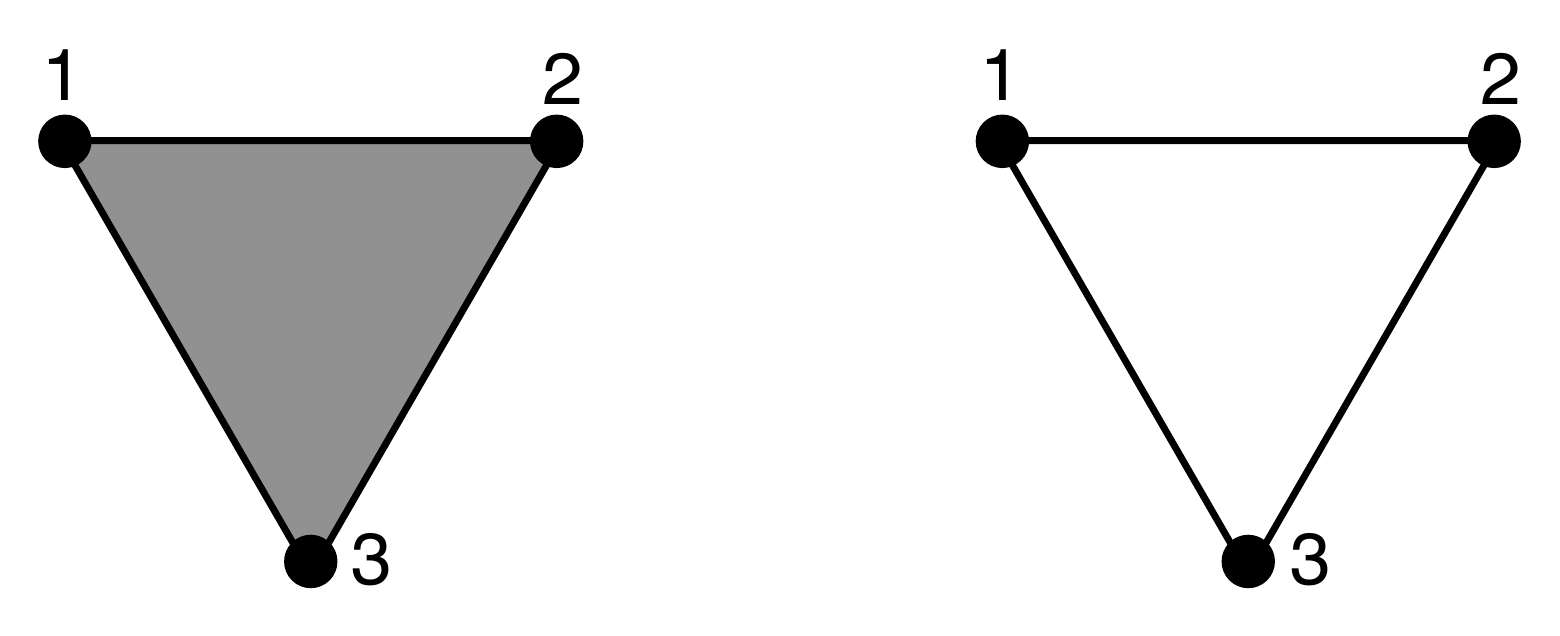}
 \end{center}
  \caption{The polyhedra of the simplicial complexes $K$ (left) and $T$ (right).}
  \label{fig:sc}
\end{figure}
 \end{ex}

\subsection{Homology}\label{sec:homology}
The procedure to define homology is summarized as follows:
\begin{enumerate}
\item Given a simplicial complex $K$, build a chain complex  $C_*(K)$. This is an algebraization of $K$ characterizing the boundary. 
\item Define homology by quotienting out  certain subspaces in $C_*(K)$ characterized by the boundary. 
\end{enumerate}

We begin with the procedure 1 by assigning orientations on simplices. 
When we deal with a $q$-simplex $\sigma=\{i_0,\dots,i_q\}$ as an ordered set, there are $(q+1)!$ orderings on $\sigma$. For $q>0$, we define an equivalence relation $i_{j_0},\dots,i_{j_q}  \sim  i_{\ell_0},\dots,i_{\ell_q} $ on two orderings of $\sigma$ such that they are mapped to each other by even permutations. 
By definition, two equivalence classes exist, and each of them is called an oriented simplex. 
An oriented simplex is denoted by $\langle i_{j_0},\dots,i_{j_q} \rangle$, and its opposite orientation is expressed by adding the minus $-\langle i_{j_0},\dots,i_{j_q} \rangle$.
We write $\langle\sigma\rangle=\langle i_{j_0},\dots,i_{j_q} \rangle$ for the equivalence class including $i_{j_0}<\dots<i_{j_q}$. For $q=0$, we suppose that we have only one orientation for each vertex. 

Let $\bF$ be a field. We construct a $\bF$-vector space $C_q(K)$ as 
\begin{align*}
C_q(K)={\rm Span}_\bF\{\langle \sigma\rangle\mid \sigma \in K_q\}
\end{align*}
for $K_q\neq\emptyset$ and $C_q(K)=0$ for $K_q=\emptyset$.
Here, ${\rm Span}_\bF(A)$ for a set $A$ is a vector space over $\bF$ such that the elements of $A$ formally form a basis of the vector space.
Furthermore, we define a linear map called the {\em boundary map} $\partial_q \colon  C_q(K)\rightarrow C_{q-1}(K)$ by 
the linear extension of 
\begin{align}\label{eq:boundary}
\partial_q\langle i_0,\dots,i_q\rangle=\sum_{\ell=0}^q(-1)^\ell\langle i_0,\dots,\widehat{i_\ell},\dots,i_q\rangle,
\end{align}
where $\widehat{i_\ell}$ means the removal of the vertex $i_\ell$. We can regard the linear map $\partial_q$ as algebraically capturing the $(q-1)$-dimensional boundary of a $q$-dimensional object. 

For example, the image of the $2$-simplex $\langle\sigma\rangle=\langle 1,2,3\rangle$ is given by 
$\partial_2\langle\sigma\rangle=\langle2,3\rangle-\langle1,3\rangle+\langle1,2\rangle$, which is the boundary of $\sigma$ (see Figure \ref{fig:sc}). 

In practice, by arranging some orderings of the oriented $q$- and $(q-1)$- simplices, we can represent the boundary map as a matrix 
\[
	M_q=(M_{\sigma,\tau})_{\sigma\in K_{q-1},\tau\in K_q}
\]
with the entry $M_{\sigma,\tau}=0,\pm 1$ given by the coefficient in \eqref{eq:boundary}. For the simplicial complex $K$ in Example \ref{exam:sc}, the matrix representations $M_1$ and $M_2$ of the boundary maps are  given by 
\begin{align}\label{eq:matrix}
M_2=\left[
\begin{array}{r}
1\\
1\\
-1
\end{array}\right],\quad
M_1=\left[
\begin{array}{rrr}
-1& 0&-1\\
 1&-1&0\\
 0& 1&1
\end{array}
\right]
\end{align}
Here the $1$-simplices (resp. $0$-simplices) are ordered by $\langle 1,2\rangle, \langle2,3\rangle, \langle 1,3\rangle$ (resp. $\langle1\rangle$, $\langle2\rangle$, $\langle3\rangle$).

We call a sequence of the vector spaces and linear maps
\begin{align*}
\xymatrix{
\cdots\ar[r] & C_{q+1}(K)\ar[r]^{\partial_{q+1}} & C_q(K)\ar[r]^{\partial_q} & C_{q-1}(K)\ar[r] & \cdots
}
\end{align*}
the {\em chain complex} $C_*(K)$ of $K$. As an easy exercise, we can show $\partial_{q}\circ \partial_{q+1}=0$ for every $q$. Hence, the subspaces $Z_q(K)={\rm ker}\partial_q$ and $B_q(K)={\rm im}\partial_{q+1}$ satisfy $B_q(K)\subset Z_q(K)$. Then, the {\em $q$th (simplicial) homology} is defined by taking the quotient space
\begin{align*}
H_q(K)=Z_q(K)/B_q(K).
\end{align*}
Intuitively, the dimension of $H_q(K)$ counts the number of $q$-dimensional holes in $K$ and each generator of the vector space $H_q(K)$ corresponds to these holes. We remark that the homology as a vector space is independent of the orientations of simplices. 

For a subcomplex $T$ of $K$, the inclusion map $\iota \colon T\hookrightarrow K$ naturally induces a linear map in homology $\iota_q\colon  H_q(T)\rightarrow H_q(K)$. Namely, an element $[c]\in H_q(T)$ is mapped to $[c]\in H_q(K)$, where the equivalence class $[c]$ is taken in each vector space. 

For example, the simplicial complex $K$ in Example \ref{exam:sc} has 
\[
Z_1(K)={\rm Span}_\bF[\begin{array}{ccc}1 & 1 & -1\end{array}]^T=B_1(K)
\]
 from (\ref{eq:matrix}). Hence $H_1(K)=0$, meaning that there are no $1$-dimensional hole (ring) in $K$.
On the other hand, since $Z_1(T)=Z_1(K)$ and $B_1(T)=0$, we have $H_1(T)\simeq \bF$, meaning that $T$ consists of one ring. Hence, the induced linear map $\iota_1\colon  H_1(T)\rightarrow H_1(K)$ means that the ring in $T$ disappears in $K$ under $T\hookrightarrow K$.

\section{Continuity of persistence diagrams of $\kappa$-complexes} \label{sec:app_stability}
We give a stability result for persistence diagrams of $\kappa$-filtrations 
which extends the stability result obtained in \cite{cdo}. 
The notation used here follows the paper \cite{cdo}. 
We first recall the definition of the Hausdorff distance
and the bottleneck distance. 
The Hausdorff distance $d_H$ on $\finite(\bR^{N})$
for $\sigma, \sigma' \in \finite(\bR^{N})$ is given by 
\[
 d_H(\sigma, \sigma') = \max \left\{\max_{x \in \sigma} \inf_{x' \in \sigma'}\|x - x'\|, \ 
 \max_{x' \in \sigma'}  \inf_{x \in \sigma} \|x - x'\| \right\}. 
\]
We define the $\ell_{\infty}$-metric on $\Delta$ by
$d_{\infty}((b_1,d_1), (b_2,d_2)) = \max(|b_1-b_2|, |d_1 - d_2|)$, where $\infty- \infty =
0$. For $(b,d) \in \Delta$, we define $d_{\infty}((b,d), \partial \Delta) =d-b$. 
For finite multisets $X$ and $Y$ in $\Delta$,
a partial matching between $X$ and $Y$
is a subset $M \subset X \times Y$ such that 
for every $x \in X$ there is at most one $y \in Y$ such that $(x,y) \in
M$ and for every $y \in Y$ there is at most one $x \in X$ such that
$(x,y) \in M$. An $x \in X$ (resp. $y \in Y$)
is unmatched if there is no $y \in Y$ (resp. $x \in X$)
such that $(x,y) \in M$. 
We say that a partial matching $M$ is
$\delta$-matching if $d_{\infty}(x, y) \le \delta$ for every $(x,y) \in M$,
$d_{\infty}(x, \partial \Delta) \le \delta$ if $x \in X$ is unmatched, 
and $d_{\infty}(y, \partial \Delta) \le \delta$
if $y \in Y$ is unmatched. 

The bottleneck distance is defined as follows
\[
 d_B(X, Y) := \inf\{\delta > 0 : \text{there exists a
 $\delta$-matching between $X$ and $Y$}\}. 
\]

For $\fp, \fp' \in \finite(\bR^{N})$ and $\kappa, \kappa' \colon  \finite(\bR^{N}) \to
[0,\infty]$, we define two complexes
 \[
  \K_{\kappa}(\fp) = \{K_{\kappa}(\fp,t)\}_{t \ge 0}, \quad 
  \K_{\kappa'}(\fp') = \{K_{\kappa'}(\fp',t)\}_{t \ge 0}. 
 \]
 Let $C$ be a \textit{correspondence} between $\fp$ and $\fp'$, i.e.,
 $C \subset \fp \times \fp'$ such that $p_1(C) = \fp$ and $p_2(C) = \fp'$, where
 $p_i$ is the projection onto the $i$th coordinate for $i=1,2$.
 We define the transpose $C^T$ of $C$, which is also a correspondence,
 by
 \[
  C^T := \{(x', x) \in \fp' \times \fp : (x, x') \in C\}. 
 \]
A correspondence $C$ defines a map from $2^\fp \setminus \emptyset$
 to $2^{\fp'} \setminus \emptyset$ as
 \[
  C(\sigma) = \{x' \in \fp' : (x, x') \in C, x \in \sigma\}. 
 \]
The distortion of $C$ is defined as
 \[
  dis(C) := \max\bigg\{
  \sup_{\sigma \subset \fp} |\kappa(\sigma) -  \kappa'(C(\sigma))|,
  \sup_{\sigma' \subset \fp'} |\kappa(C^T(\sigma')) -
  \kappa'(\sigma')| \bigg\}.  
 \]
\begin{lem}\label{lem:C2}
 If $dis(C) \le \eps$, then $H_q(\K_{\kappa}(\fp))$ and
 $H_q(\K_{\kappa'}(\fp'))$ are $\eps$-interleaving. 
\end{lem}
\begin{proof}
 Assume that $\sigma \in K_{\kappa}(\fp,t)$ and $\kappa(\sigma) \le t$.
 Then it follows from $|\kappa(\sigma) - \kappa'(C(\sigma))| \le \eps$ that
 \[
 \kappa'(\sigma') \le \kappa'(C(\sigma)) \le \kappa(\sigma)  + \eps
 \le t+\eps  \text{ for any $\sigma' \subset C(\sigma)$},
 \]
 which implies $\sigma' \in K_{\kappa'}(\fp', t + \eps)$, and hence
 $C$ is $\eps$-simplicial from $\K_{\kappa}(\fp)$ to
 $\K_{\kappa'}(\fp')$. Symmetrically, $C^T$ is also
 $\eps$-simplicial.
 Therefore, the conclusion follows from Proposition~4.2 in \cite{cdo}. 
\end{proof}

Let us define
\begin{equation}\label{defn-of-S}
 S((\kappa,\fp), (\kappa', \fp'))
 := \sup_{\sigma \subset \fp, \sigma' \subset \fp' \atop{d_H(\sigma,
 \sigma') \le d_H(\fp, \fp')}} |\kappa(\sigma) - \kappa'(\sigma')|. 
\end{equation}
We remark that $S((\kappa,\fp), (\kappa', \fp')) = \|\kappa -
\kappa'\|_{\infty}$ if $\fp=\fp'$.

 \begin{lem}\label{lem:C1}
  Let $C$ denote the correspondence defined by 
 $C = \{(x, x') \in \fp \times \fp' : \|x - x'\| \le
  d_H(\fp,\fp')\}$. Then,
  \[
   dis(C) \le S((\kappa, \fp), (\kappa', \fp')). 
  \]
 \end{lem}
\begin{proof}
We easily see that
\[
\sup_{\sigma \subset \fp} d_H(\sigma, C(\sigma)) \le d_H(\fp,\fp') 
~~{\rm and}~~
\sup_{\sigma' \subset \fp'} d_H(C^T(\sigma'), \sigma') \le d_H(\fp,\fp')
 \]
which implies the assertion.
\end{proof}

 For $D_q(\kappa, \Xi) = D_q( \K_{\kappa}(\Xi) ) $ and 
 $D_q(\kappa', \Xi') = D_q( \K_{\kappa'}(\Xi') ) $, we obtain the following
 continuity result. 
 \begin{thm}
\begin{equation}
d_B(D_q(\kappa,\fp), D_q(\kappa', \fp')) \le S((\kappa,\fp), (\kappa',
 \fp')).
 \label{eq:continuity}
\end{equation} 
 \end{thm}
 \begin{proof}
It follows from Lemma~\ref{lem:C1} and Lemma~\ref{lem:C2} that 
$H_q(\K_{\kappa}(\fp))$ and $H_q(\K_{\kappa'}(\fp'))$ are $S((\kappa,\fp),
  (\kappa', \fp'))$-interleaving. Therefore, we obtain
  \eqref{eq:continuity} from \cite{cdgo}. 
 \end{proof}

 \begin{cor}
  Suppose that $\kappa$ is Lipschitz continuous with respect to
  $d_H$, i.e., there exists a constant $\gamma>0$ such that 
 \[
  |\kappa(\sigma) - \kappa(\sigma')| \le \gamma d_H(\sigma, \sigma')
  \text{ for $\sigma, \sigma' \in \finite(\bR^{N})$}. 
\]
Then, 
\begin{equation}
 d_B(D_q(\kappa,\fp), D_q(\kappa, \fp')) \le \gamma d_H(\fp,\fp'). 
 \label{eq:continuity2}
\end{equation} 
 \end{cor}
\begin{proof}
 From the assumption and \eqref{defn-of-S}, we see that 
 \[
 S((\kappa, \fp), (\kappa, \fp')) \le \gamma d_H(\fp, \fp'). 
 \]
 Therefore, \eqref{eq:continuity2} follows. 
\end{proof}

\section*{Acknowledgment}
The authors would like to thank Ippei Obayashi and Kenkichi Tsunoda for useful discussions. 
All authors are partially supported by JST CREST Mathematics (15656429). 
T.S. is partially supported by JSPS Grant-in-Aid (26610025, 26287019). T.K.D. is partially supported by JSPS Grant-in-Aid for Young Scientists (B) (16K17616).

\bibliographystyle{acmtrans-ims}
\bibliography{bpp}

\end{document}